\documentclass{amsart}
\usepackage{bbm}
\usepackage{amsmath,amsthm,amsfonts,amssymb,amscd,mathrsfs}
\usepackage{psfrag,graphicx}

\usepackage{epic,eepic}
\usepackage{color}
\usepackage{ebezier}

\theoremstyle{plain}
\newtheorem{main}{Theorem}

\newtheorem{maincor}[main]{Corollary}
\newtheorem{Thm}{Theorem}[section]
\newtheorem{Lem}[Thm]{Lemma}
\newtheorem{Prop}[Thm]{Proposition}
\newtheorem{Cor}[Thm]{Corollary}

\theoremstyle{remark}

\newtheorem{Rem}[Thm] {Remark}

\newtheorem*{Que}{Question}

\long\def\begcom#1\endcom{}

\newcommand{\Vol}{\operatorname{vol}}

\newcommand{\diam}{\operatorname{diam}}

\newcommand{\length}{\operatorname{\length}}

\newcommand{\Diff}{\operatorname{Diff}}

\def\Diff{\operatorname{Diff}}

\def\id{\operatorname{id}}
\def\length{\operatorname{length}}
\def\loc{\operatorname{loc}}

\def\spec{\operatorname{sp}}

\def\Int{\operatorname{Int}}

\def\NN{\mathbb{N}}

\def\cF{{\mathcal F}}
\def\cM{{\mathcal M}}

\def\vep{\varepsilon}

\def\NN{{\mathbb N}}

\def\Diff{\operatorname{Diff}}

\def\id{\operatorname{id}}
\def\length{\operatorname{length}}
\def\loc{\operatorname{loc}}

\def\spec{\operatorname{sp}}

\def\NN{\mathbb{N}}

\def\cF{{\mathcal F}}
\def\cM{{\mathcal M}}

\def\vep{\varepsilon}

\def\NN{{\mathbb N}}

\begin{document}

\title[Asymptotic $h$-expansiveness rate of $C^\infty$ maps  ]
      {Asymptotic $h$-expansiveness rate of $C^\infty$ maps}

\author[D.Burguet, G.Liao, J. Yang]
{ David Burguet$^{1}$, Gang Liao$^{2}$,  and  Jiagang Yang$^{3}$}

\email{david.burguet@upmc.fr}

\email{liaogang@math.pku.edu.cn}

\email{yangjg@impa.br}



\thanks{$^{1}$ LPMA - Universit Paris 6, 4, Place Jussieu, 75252 Paris Cedex 05 France.}

\thanks{$^{2}$ School of Mathematical
Sciences,  Peking University, Beijing 100871, China. GL would like
to thank the kind hospitality of IMPA and the support of
CNPq-Brazil.}

\thanks{$^3$
Departamento de Geometria, Instituto de Matem\'atica e
Estat\'istica, Universidade Federal Fluminense, Niter\'oi, Brazil.
JY is partially supported by CNPq and FAPERJ}

\date{April, 2014}

 \maketitle



\begin{abstract}We study the rate of convergence to zero of the tail entropy of $C^\infty$ maps.
 We give an upper bound of this rate in terms of the growth in $k$ of the derivative of order $k$ and give  examples showing
the optimality of the established rate of convergence. We also
consider the case of multimodal maps of the interval. Finally we
prove that homoclinic tangencies give rise to $C^r$ $(r\geq 2)$
robustly non $h$-expansive dynamical systems.
\end{abstract}

\tableofcontents

\section{Introduction}
\emph{Topological Entropy.}
A dynamical system $(f,M)$ is defined by a continuous map $f: M\to
M$  on a compact topological space $M$.  The topological entropy $h(f)$ of $(f,M)$
 introduced by Adler, Konheim and McAndrew \cite{AKM} estimates the dynamical complexity of the system by counting
 the exponential growth rate of distinguishable orbits at arbitrarily small scales.
 The topological entropy is a topological invariant, i.e. it is invariant under topological conjugacy.
 In this pioneer work \cite{AKM} the authors use finer and finer open covers as the decreasing scale to define the topological
 entropy.
 Later Bowen \cite{Bow71} gave an equivalent definition for metric spaces  $M$ with distance $d$ (in the present paper we will only consider
  $C^\infty$ smooth manifolds $M$ endowed with a Riemannian metric). Let us recall Bowen's definition.

   For any subset $\Lambda\subset
M$,  $\vep>0$ and  $n\in \mathbb{N}$, a subset $K\subset
 M$ is said $(n,\vep)$-spanning $\Lambda$ if for any $x\in
 \Lambda$ there exists $y\in K$ such that  $d(f^ix,f^iy)<
 \vep$ for $i=0,1,...,n-1$. Let
 $r_n(f,\Lambda,\varepsilon)$ denote the smallest cardinality of any
$(n,\vep)$-spanning set of $\Lambda$. The $\vep$-topological entropy
of $\Lambda$ is defined by
$$h_d(f,\Lambda,\varepsilon)=\limsup_{n\rightarrow\infty}\frac{1}{n}\log
r_n(f,\Lambda,\varepsilon).$$ Letting $\vep\to 0$, define the
topological entropy of $f$ on $\Lambda$ by
$$h_d(f,\Lambda)=\lim_{\varepsilon\rightarrow0}h_d(f,\Lambda,\varepsilon).$$
Denote $h_d(f,\varepsilon)=h_d(f,M,\varepsilon)$ and
$h_d(f)=h_d(f,M)$. By an easy
argument of compactness one then can prove $h_d(f)$ is equal to the
topological entropy as defined in \cite{AKM} by using open covers.
In particular $h_d(f)=h(f)$ does not depend on the metric $d$.
 However this is not  the case of $h_d(f,\vep)$. If $d_1$ and $d_2$ are two equivalent metrics
 then there exists $C>1$ such that $h_{d_2}(f,C\vep)\leq h_{d_1}(f,\vep)\leq h_{d_2}(f,C^{-1}\vep)$ for all $\vep>0$.
In the present paper we  endow compact smooth manifolds with  Riemannian metrics.
As such metrics are equivalent, the $\vep$-entropy of $f$ is well defined up to some constant $C>1$ as above. From now the distance $d$ on $M$ is fixed and we forget the index $d$ in the above definitions.\\

\emph{Tail entropy and $h$-expansiveness.} The tail entropy $h^*(f)$
of a topological system $(f, M)$ first
 appeared in \cite{Mis76} (initially Misiurewicz called it topological conditional entropy).
 It is the entropy remaining at arbitrarily small scales. The tail entropy bounds the default
  of upper semi-continuity of the entropy of invariant Borel probability measures (see \cite{Walter}
  for the entropy of invariant measures). This property established in \cite{Mis76} is certainly
  the main motivation to consider this quantity.  As for the topological entropy Bowen gave a
  definition of the tail entropy for metric spaces replacing iterated open covers by dynamical balls in the definition of Misiurewicz.
  We present two equivalent definitions.

 Given $x\in M$, $n\in \NN$,
denote the $n$-step dynamical ball $B_{n}(f, x,\varepsilon)$
consisting of all such points $y\in M$ that
$$d(f^iy,\,f^ix)<\varepsilon,~~i=0, 1,\cdots,n-1.$$
Define the  upper $\vep$-tail entropy $\overline{h}^*(f,\vep)$ as
follows :

$$\overline{h}^*(f,\vep)=\lim_{\delta\to 0}\limsup_{n\rightarrow\infty}\frac{1}{n}\log
\sup_{x\in M}r_n(f,B_n(f,x,\vep),\delta).$$ The lower  $\vep$-tail
entropy $\underline{h}^*(f,\vep)$ is the maximal entropy of
infinite dynamical balls. More precisely  let
$B_{\infty}(f,x,\vep)=\bigcap_{n\in \mathbb{N}}B_{n}(f,
x,\varepsilon)$. Define the lower $\vep$-tail entropy as follows :
$$\underline{h}^{*}(f,\varepsilon)=\sup_{x\in M}h(f,\, B_{\infty}(f,x,\varepsilon)).$$

One easily finds that $\underline{h}^{*}(f,\varepsilon)\leq
\overline{h}^{*}(f,\varepsilon)$ for all $\vep>0$. Moreover by an
argument of compactness Bowen (Proposition 2.2 of \cite{Bow72b}) has
shown that for all $\vep>0$ we have in fact the equality
 $\overline{h}^{*}(f,\vep)=
\underline{h}^{*}(f,\varepsilon)$  and we denote from now on this
quantity by $h^*(f,\vep)$. Also we can define the tail entropy
$h^*(f)$ as follows :
$$h^*(f)=\lim_{\vep\to 0}h^*(f,\vep).$$

Like the topological entropy,  the tail entropy is a topological
invariant and thus $h^*(f)$ does not depend on the metric $d$,  however $h^*(f,\vep)$
 may depend on $d$. But as already noted this is not important in our smooth setting up to
 rescale balls by a uniform constant.

The dynamical system $(f,M)$ is called entropy expansive
($h$-expansive) when there exists $\vep>0$ such that
$h^{*}(f,\vep)=0$ and asymptotically  entropy expansive
(asymptotically $h$-expansive) when $h^{*}(f)=0.$ As noticed above
the measure theoretical entropy  is upper semi-continuous for
asymptotically $h$-expansive maps and therefore such maps always
admit an invariant measure of maximal entropy.

The notion of $\varepsilon$-tail entropy is broadly used in the calculation of
 entropy, since  by Theorem 2.4 of \cite{Bow72b} it bounds the difference of $\vep$-entropy
and the whole entropy\footnote{However this inequality is in general quite rough and both members may have a different order of magnitude, even for asymptoically $h$-expansive systems. See Proposition \ref{rrrem}.}:
\begin{eqnarray}\label{rer}
|h(f)-h(f,\vep)|\leq h^{*}(f,\vep).\end{eqnarray} For any
$f$-invariant Borel probability measure $\mu $ and for any finite
Borel partition $P$ with diameter less than $\vep$ we have also
\begin{eqnarray}\label{rerr}
|h(\mu)-h(\mu,P)|\leq h^{*}(f,\vep).\end{eqnarray}

We present now another notion introduced by Newhouse in \cite{New89}
as the $\vep$-local entropy. We first define a notion of local
entropy for invariant measures. Let $\mu$ be an $f$-invariant probability
measure and let $\vep>0$ we put

$$h_{\loc}(\mu,\vep):=\lim_{1\neq \sigma \rightarrow 1}\inf_{F,\ \mu(F)\geq \sigma}\lim_{\delta\to 0}\limsup_{n\rightarrow\infty}\frac{1}{n}\log
\sup_{x\in F}r_n(f,F\cap B_n(f,x,\vep),\delta).$$

In \cite{burfis} it is shown like for the  tail
entropy that $h_{\loc}(\mu,\vep)$ may be written by using infinite dynamical balls as follows
$$h_{\loc}(\mu,\vep)=\lim_{1\neq \sigma \rightarrow 1}\inf_{F,\ \mu(F)\geq \sigma}\sup_{x\in F} h(f,F\cap B_{\infty}(f,x,\vep)).$$
Finally we let $h_{\loc}(f,\vep):=\sup_\mu h_{\loc}(\mu,\vep)$ be
the $\vep$-local entropy of $f$. Clearly we have $h_{\loc}(f,\vep)
\leq h^*(f,\vep)$. Moreover by the so called tail variational
principle proved in \cite{Dows}, $\lim_{\vep\rightarrow 0}
h_{\loc}(f,\vep)=h^*(f)$. However we do not know if
$h_{\loc}(f,\vep)=h^*(f,\vep)$ (or even  $\geq h^*(f,\vep/10)$) for
any $\vep>0$. Newhouse proved ( Theorem 1.2 of \cite{New89}) that
$h_{\loc}(f,\vep)$ also satisfies Inequality (\ref{rer}) and the
following estimate for the entropy of measures finer than Inequality
(\ref{rerr}):

$$|h(\mu)-h(\mu,P)|\leq h_{\loc}(\mu,\vep).$$

The $\vep$-local entropy is defined through invariant measures and we do not know if it can be expressed in a topological way.
 Conversely we ignore any satisfactory measure quantity $h^*(\mu,\vep)$ such that a variational
  principle $h^*(f,\vep)=\sup_\mu h^*(\mu,\vep)$ holds and  such that  $\left(h^*(.,\vep)\right)_\vep$
  defines an entropy structure (See \cite{Dows} for the theory of entropy structures).\\

\emph{Local volume growth.}
We introduce now the local volume growth which is closely related with the local entropy. We assume here that $f$ is $C^r$ with $r\geq 1$.

A $C^r$ map $\sigma$ from an open set $U\supset [0,1]^k$ of
$\mathbb{R}^k$ to $M$, which is a diffeomorphism onto its image, is
called a  $k$-disk. For any $k$-disk $\sigma$ and for any Borel subset $E$ of $[0,1]^k$
 we  denote by $\left|\sigma|_E\right|$ the $k$-volume of $\sigma$ on $E$, i.e. $\left|\sigma|_E\right|=\int_{E}\|\Lambda^kD_t\sigma\|_kd\lambda(t)$
 where $d\lambda$ is the Lebesgue measure on $[0,1]^k$. Then for any $\vep>0$ we define
 the $\vep$-local  $k$-volume growth $v_k^*(f,\vep)$ of $f$  as  follows :
$$v_k^*(f,\vep)=\sup_{\sigma, \ k\mathrm{-disk}}\limsup_{n\to\infty}\frac{1}{n}\sup_{x\in M}\log\left| f^{n-1}\circ
\sigma|_{ \sigma^{-1}\left(B_n(f,x,\vep)\right)}\right|.$$  By using
Pesin theory Newhouse \cite{New89}  proved that for $C^{1+\alpha}$
dynamical systems,
\begin{eqnarray*}
h_{\loc}(f,\vep)&\leq &\max_k v_k^*(f,2\vep)
\end{eqnarray*}
 where $k$ takes over all numbers not more than the center unstable  dimensions of ergodic invariant measures with positive entropy.
 For  surface diffeomorphisms with nonzero topological entropy the only possible
 value of $k$ is one by Ruelle inequality \cite{Ruelle}, i.e. in this case we have
\begin{eqnarray}\label{surf} h_{\loc}(f,\vep)&\leq &
v^*_1(f,2\vep).
\end{eqnarray}
Finally let us define the local $k$-volume growth $v^*_k(f)$ of $f$
for any $1\leq k\leq \dim(M)$ as
$$v^*_k(f)=\lim_{\vep \to 0}v^*_k(f,\vep).$$\\

\emph{Yomdin's entropy theory of $C^r$ smooth maps.}
 In \cite{Yom87} Yomdin introduced semi-algebraic tools to study the local complexity of smooth maps
and proved in this way Shub's entropy conjecture for $C^\infty$
maps. This famous conjecture \cite{Shub} states that the topological
entropy $h(f)$ has always
   the logarithm of the spectral radius $\spec(f)$ in homology   as a lower bound
  for differentiable maps.  It follows from the inequality
  $ \log \spec(f)\leq h(f)+ \max_{1\leq k\leq \dim(M)}v^*_k(f)$ together with  the following estimate on the local volume
  growth established in \cite{Yom87} for any $C^r$ map $f$:
\begin{eqnarray}
\label{ r upper bound }v^*_k(f)&\leq &\frac{kR(f)}{r},
\end{eqnarray}
with  $R(f)=\lim_n\frac{1}{n}\log^+ \sup_{x\in
M}\|D_xf^n\|$.
 Loosely speaking, the larger
the differential order is, the more regular the dynamical complexity
is. Using the estimate (\ref{ r upper bound }), Yomdin \cite{Yom87}
and Newhouse \cite{New89} in the setting of $C^{\infty}$ maps showed
entropies in both topological and measure theoretic sense are upper
semi-continuous.

Later  Buzzi \cite{Buz97} further observed that  Yomdin's work in
fact implied directly (without refering to local volume growth) that
\begin{eqnarray}\label{buzz}h^*(f)\leq \frac{\dim(M)R(f)}{r}.\end{eqnarray}
Consequently,  all $C^\infty$  maps are asymptotically $h$-expansive.\\

\emph{Misiurewicz-like examples.} In the early seventies Misiurewicz
\cite{Mis73}  produced $C^r$ diffeomorphisms $f_r$ without any
measure of  maximal entropy for any finite $r$, in particular
$h^*(f_r)\neq 0$. In fact in this example one can compute
$$h^*(f_r)\geq \frac{R(f)}{r}$$ (it corresponds to the converse
inequality of (\ref{buzz}) up to the factor $\dim(M)$). The main
idea consists in accumulating smaller and smaller horseshoes at a
periodic point which admits a homoclinic tangency. Later Buzzi
 \cite{Buz97} built in the same spirit a $C^r$
interval map with $h^*(f)= \frac{R(f)}{r}$ and then by considering
the product of such systems one can see that inequality (\ref{buzz})
 is sharp for noninvertible maps.  See also \cite {DN05}, \cite{Buz13} for related recent works.\\

\emph{Rate of convergence of the tail entropy for $C^\infty$
systems.} As stated above $C^\infty$ systems are asymptotically
$h$-expansive, i.e. $\lim_{\vep\to 0}h^*(f,\vep)=0$. This paper is
devoted to the study of the rate of convergence in the previous
limit. This was first investigated by Yomdin in \cite{Yom91} for
 analytic surface  diffeomorphisms. He proved by using ``analytic unit
reparametrization of semi-algebraic sets" via Bernstein inequalities
that $$h_{\loc}(f,\vep)\leq h^*(f,\vep)\leq C(f)\frac{\log|\log
\vep|}{|\log\vep|}$$  for any $\vep>0$ and for some constant $C(f)$
depending only on $f$. More recently Liao \cite{Liao} proved that
for any compact analytic manifold $M$, there exists  a universal
function $a:\mathbb{R}^+\to\mathbb{R}^+$ with $\lim_{\vep\to 0}
a(\vep)=0$ such that for any analytic map $f$ on $M$,  the
$\vep$-tail entropy satisfies
$$h^*(f,\vep)\leq C(f)\,a(\vep)$$
for some constant $C(f)$, independent of $\vep$.
 Here we investigate the case of general $C^\infty$ maps and give an explicit rate of convergence
in terms of the growth in $k$ of the supremum norms of the
derivatives of order $k$ by using ``$C^k$ unit reparametrizations of
semi-algebraic set" as in the proof of the entropy conjecture by
Yomdin \cite{Yom87}.  In the same spirit of the previously mentioned
sharp $C^r$ examples we will then produce various examples, proving
optimality of the established rate of convergence. We precise
moreover as in \cite{Yom91} the modulus of upper semicontinuity of
the topological entropy for some dynamical systems.\\

\emph{$h$-expansiveness and homoclinic tangencies.} Hyperbolic
systems are known to be expansive and therefore $h$-expansive. In
fact they are robustly ($h$-)expansive for the $C^1$ topology, i.e.
for any hyperbolic system  $f$ there exists $\vep>0$ and a $C^1$
neighborhood $\mathcal{U}$ of $f$ such that $h^*(g,\vep)=0$ for any
$g \in U$.  Note that any $C^1$ robustly expansive diffeomorphism is
Axiom A  as shown by  Ma\~n\'e \cite{Mane}.
 For interval maps hyperbolicity is $C^r$ open and dense for any $r$ \cite{Koz}
 but the celebrated Newhouse phenomenon claims this is no more the case for diffeomorphisms
 in higher dimensions \cite{New79}.   In \cite{LVY} Liao, Viana and Yang proved that any diffeomorphism $C^1$
 far from homoclinic tangencies is $C^1$-robustly $h$-expansive.
 In Theorem \ref{homone} and Theorem \ref{non EE}
 we prove that $C^2$ interval maps and diffeomorphisms in higher dimensions with a non degenerate homoclinic tangency
   are not $C^2$-robuslty $h$-expansive, which gives somehow a reverse to the previous result of \cite{LVY}. \\

To resume non $h$-expansiveness of  smooth map is produced
by homoclinic tangencies while the rate of convergence of the $\vep$-tail entropy
is related with the growth of higher derivatives.

\section{Statements of results}\label{state}

\subsection{Explicit rate for ultradifferentiable maps}
\subsubsection{Ultradifferentiable maps}
An arbitrary sequence of positive real numbers $\cM=
(M_k)_{k\in\NN}$ with $M_0\geq 1$  will always be called a weight
sequence. A quite usual  condition on the weight is the logarithmic convexity.
A weight sequence $(M_k)_k$ is called logarithmic convex if for all
$k\in \NN\setminus\{0\}$ we have $$2\log M_k \leq \log M_{k+1} + \log M_{k-1}.$$
 We will use two important properties of logarithmic convex weights :
\begin{itemize}
\item $\left(\left(M_k/M_0\right)^{1/k}\right)_k$ is nondecreasing;\\
\item $M_kM_l\leq M_0M_{k+l}$ for all $k,l\in\NN$.
\end{itemize}
A  weight  $\mathcal{M}=(M_k)_k$ is called  superexponential when  $(M_k)_k$ satisfies $$\liminf_k\frac{\log\left(M_k/M_0\right)}{k}=+\infty.$$
Note that logarithmic convex weights are quite general: if a weight
$(M_k)_k$ is superexponential then there exists a
logarithmic convex weight $(M'_k)_k$ with $M'_k\leq M_k$ and
$M_l=M'_l$ for infinitely many $l\in \NN$.\\

 Let $O\subset\mathbb{R}^s$ be an open set and let
$\cM=(M_k)_k$ be  a weight sequence,  we define the set
$U^\cM(O,\mathbb{R}^t)$  and its subset $V^\cM(O,\mathbb{R}^t)$ of respectively $U$- and $V$-ultradifferentiable maps
with respect to $\cM$ as follows
$$U^\cM(O,\mathbb{R}^t)=\Big{\{}f=(f_1,\cdots,f_t)\in C^\infty(O,\mathbb{R}^t): \ \exists h>0,\,\, s.t.\,\,  \max_{i=1,\cdots,t}\sup_{r\in\NN}\frac{\|D^rf_i\|_\infty}{h^rM_r}<\infty\Big{\}},$$
and
$$V^\cM(O,\mathbb{R}^t)=\Big{\{}f=(f_1,\cdots,f_t)\in C^\infty(O,\mathbb{R}^t): \ \max_{i=1,\cdots,t}\|D^rf_i\|_\infty\leq M_r\Big{\}}.$$

The previous setting of ultradifferentiable maps is well adapted to
the characterization of quasi-analytic maps. An ultradifferentiable
class  $U/V^\cM$   is said quasi-analytic if there is no
nontrivial function in  $U/V^\cM$ with compact support. The famous
Denjoy-Carleman theorem claims that an ultradifferentiable class
$U/V^\cM$ is quasi-analytic if and only if
$$\sum_{k\in \NN} \frac{1}{\inf_{j\geq k} M_j^{1/j}}<+\infty.$$

In the present paper we study the entropy of smooth maps on a
compact  smooth Riemannian manifold $M$ of dimension $m$.  We
consider the exponential map $exp$ associated to the smooth
Riemannian metric on $M$ and we denote by $R_{inj}$ its radius of
injectivity. The first derivative is the important map in the
estimation of the entropy. Given a weight $\cM$ we work in the
following on  the spaces $C_{U/V}^\cM(M)$ (resp.
$\Diff_{U/V}^\cM(M)$) of $C^\infty$ maps (resp. diffeomorpshims)
$f:M \rightarrow M$ whose first derivative is
$U/V$-ultradifferentiable with respect to $\cM$ through the local
charts given by the exponential map, i.e.  for any $x\in M$ and for
any $R$, satisfying $f(B_M(x,R))\subset B_M(f(x),R_{inj})$, we have
$$D(exp_{f(x)}^{-1}\circ f\circ exp_x)\in U/V^\cM(B_{\mathbb{R}^m}(0,R),\mathbb{R}^{m^2}).$$

\subsubsection{Algebraic Lemma}
We present in this section the main semi-algebraic tool used in
Yomdin's entropy theory.

For any integer $r$ and for any $C^r$ map we let $\|f\|_r$ be the
supremum norm of the derivatives of order no more than $r$:
$$\|f\|_r:=\max_{k=1,\cdots,r}\|D^kf\|.$$

The following algebraic lemma was stated by Gromov in \cite{Gr85}.\\

\noindent{\bf Algebraic Lemma. } \,\,\label{ygpol} Let $P\in
(\mathbb{R}[X_1,\cdots,X_l])^m$ be a real $m$-vector polynomial in
$l$ variables of total degree $r$. Then there exists an integer
$C_{r,l,m}$ depending only on $r$ and $m$ (but not on the coefficients
of $P$) and continuous maps
$\phi_1,\cdots,\phi_{C_{r,l,m}}:[0,1]^l\rightarrow [0,1]^l$, such that
:
\begin{itemize}
\item $P^{-1}([0,1]^m)=\bigcup_{i=1,\cdots,C_{r,l,m}}\phi_i([0,1]^l)$;\\
\item $\phi_i$ is analytic on $(0,1)^l$ for each $i$;\\
\item $\|P\circ \phi_i\|_r\leq 1$ and $\|\phi_i\|_r\leq 1$ for each $i$.
\end{itemize}

A complete proof of this lemma may be found in \cite{Bur08} or \cite
{PW06}.\\

We need to estimate the algebraic complexity $C_{r,l,m}$ in the
previous lemma. In Section \ref{quantgr} we are going  to show
$C_{r,1,m}$ grows polynomially with $r$.

\subsubsection{Main result}Let us first set some notations.

When  $a=(a_k)_{k\in \NN}$ is a non decreasing unbounded sequence
of positive real numbers with $a_0=0$, we will consider the inverse function $a^{-1}$ of
$a$ defined for all nonnegative real numbers  $x$ by
$$a^{-1}(x):=\sup \{l\in \NN, \ a_l\leq x\}\in \NN.$$

The inverse function $a^{-1}$ of $a$ is  an  non decreasing unbounded function on $\mathbb{R}^+$.
Observe also that  if $a=(a_k)_{k\in \NN}$  and   $b=(b_k)_{k\in \NN}$ are two sequences as above with
$a_k\geq b_k$ for all $k\in \NN$, then $a^{-1}(x)\leq b^{-1}(x)$ for all $x\geq 0$.

For a logarithmic convex superexponential weight $\mathcal{M}=(M_k)_k$ we denote by $G_\mathcal{M}$
the inverse function of $a^\cM=(a_k^\cM)_k$ with $a_0^\cM=0$ and $a_k^\cM=\frac{\log^+ (M_k/M_0)}{k}$ for $k> 0$.

For integers $0\leq l\leq m$ and for a  real number $D\geq 1$, we call a weight $\mathcal{M}=(M_k)_k$
 $(l,m,D)$-admissible when  $M_0\geq e$ and for all integers $k>0$ :
$$\frac{\log (M_k/M_0)}{k}\geq \log k+\frac{2k\log (2^{2m+l}C_{k,l,m}k^{2l})}{Dl},$$
where$C_{k,l,m}$ is the constant in the Algebraic Lemma. A weight is said $(l,m)$-admissible (resp. admissible) if it is $(l,m,D)$-admissible for some $D$ (resp. for some $l,m,D$). Admissible  weights are superexponential. \\

For an $U$-ultradifferentiable map $f\in C_{U}^\cM(M)$ with respect to a
logarithmic convex admissible weight $\cM$ the rates of convergence to zero of
the $\vep$-tail entropy, $h^*(f,\vep)$, and of the $\vep$-local
volume growths, $(v_l^*(f,\vep))_{l\leq m}$,  are related with the growth in $r$ of $M_r$ as follows.

\begin{main}\label{gen}
Let $M$ be a  compact smooth Riemannian manifold of dimension $m$,
$0\leq l\leq m$ be an integer, $D$  be a positive real number and
$\mathcal{M}=(M_n)_n$ be an $(l,m,D)$-admissible (resp.
$(m,m,D)$-admissible) weight.

Then for all $f\in C_V^{\mathcal{M}}(M)$  and for all $0<\vep<\min(1,R_{inj}^{2})$, we have
$$v_l^*(f,\vep)\leq \frac{(2D+1)l}{ G_\mathcal{M}\left(|\log \vep|/2\right)}\log
M_0 \ \left(\text{resp. } h^*(f,\vep)\leq \frac{(2D+1)m}{ G_{\mathcal{M}} \left(|\log \vep|/2\right)}\log M_0\right).$$
\end{main}

 If $f$ is in $C_U^\cM(M)$ for some  logarithmic superexponential weight  $\mathcal{M}=(M_k)_k$ then $f$ is in $C_V^{\tilde{\cM}}(M)$ for the logarithmic convex weight $\tilde{M}=(\tilde{M}_k)_k$ with $(\tilde{M}_k)_k=(ab^kM_k)_k$ for some constants $a$ and $b$ depending on $f$. Then one easily sees that there exists a constant $C=C(a,b)$ such that  for all $x\geq C$ we have
$$G_{\tilde{\mathcal{M}}}(x)\geq G_{\mathcal{M}}(x-C).$$
Therefore we get the following estimates for $U$-ultradifferentiable classes:

\begin{maincor}\label{oups}
Let $M$ be a  compact smooth Riemannian manifold of dimension $m$,
$0\leq l\leq m$ be an integer and  $\mathcal{M}=(M_n)_n$ be an
$(l,m)$-admissible weight (resp. $(m,m)$-admissible).

Then for all $f\in C_U^{\mathcal{M}}(M)$, there exists a constant $C=C(f,\mathcal{M})\geq 1$, such that   for all $0<\vep<1/C$ we have

$$v_l^*(f,\vep)\leq \frac{C}{ G_{\mathcal{M}}\left(|\log (C\vep)|/2\right)} \ \left(\text{resp. } \, h^*(f,\vep)\leq \frac{C}{ G_{\mathcal{M}}\left(|\log (C\vep)|/2\right)}\right).$$
\end{maincor}

Since $C_{r,1,m}$ grows polynomially with $r$ as shown in Proposition \ref{grone}, the weight $(k^{k^2})_k$
is $(1,2)$-admissible. Together with Inequality (\ref{surf}) we get as a consequence

\begin{maincor}\label{surfh}
Let $M$ be a compact smooth Riemannian surface. Then for all $f\in
\Diff_U^{(k^{k^2})_k}(M)$,  there exists a constant $C=C(f)\geq 1$,
such that   for all $0<\vep<1/C$ we have
$$h_{\loc}(f,\vep)\leq \frac{C\log|\log \vep|}{|\log \vep|}.$$
\end{maincor}

Analytic maps corresponds to the $U$-ultradifferentiable class with respect to the weight $\cM=(k^k)_k$. In
particular the above Corollary applies to analytic maps. We get in
this way a new proof of Yomdin's result \cite{Yom91} for the
$\vep$-local entropy with a real approach, i.e. by using $C^k$
reparametrizations instead of analytic unit reparametrizations  of
semi-algebraic sets.  However Corollary \ref{surfh} is more general
 as it states that analytic maps are not the largest
$U$-ultradifferentiable class for which the rate
 in  $\frac{\log|\log\vep|}{|\log\vep|}$
applies.

\subsection{Rate for multimodal maps of the interval}

We consider in this section multimodal maps, i.e. continuous
 piecewise (with a finite number of pieces) monotone maps of
  the unit interval. Such a map $f:[0,1]\rightarrow [0,1]$ is said $l$-multimodal
  if  $[0,1]$ may be subdivided into $l$ and not more intervals of monotonicity for $f$.
  We also let $L(f)$ be the least length of any subinterval of this minimal partition.

It was proved by Misiurewicz  and Szlenk that multimodal maps are
asymptotically $h$-expansive \cite{Misch}. We can in fact give a
precise estimate of the rate of entropy of the $\vep$-tail entropy for smooth multimodal maps.

\begin{main}\label{wmulti}
Let $f$ be a $C^1$ multimodal map. Then we have for any $0<\vep<L(f)$

$$h^*(f,\vep) \leq \frac{\log2\cdot\log^+
\|f'\|_{\infty}}{\log^+\left(1/ w(f',\varepsilon)\right)}$$ where
$w(f',.)$ is the modulus of continuity of $f'$, defined for any
$\vep>0$ as $w(f',\vep):=\sup_{|x-y|<\varepsilon}|f'(x)-f'(y)|$.
\end{main}

For $C^l$ $l$-multimodal maps $f$ one can get  an upper bound for any $\vep$,
independently from $L(f)$.

\begin{main}\label{quasionedim}
Let $f$ be a $C^l$ $l$-multimodal map. Then we have for any $0<\vep<1$
$$h^*(f,\vep) \leq \frac{\log^+ \|f\|_{l}}{|\log
\vep|}.$$
\end{main}

 By Markov inequality, for any integer $r>0$ there
exists a constant $C(r)$ such that, for any polynomial
$P:[0,1]\rightarrow [0,1]$ of degree $r$, one has $\|P\|_{r}\leq
C(r)$. Therefore, Theorem \ref{quasionedim} yields
\begin{maincor}\label{polynomial} For any polynomial
$P:[0,1]\rightarrow [0,1]$ of degree $r$, we have for any $0<\vep<1$
$$h^*(P,\vep)\leq
\frac{\log^+ C(r)}{|\log \vep|}.$$
\end{maincor}

We will show in the next section that the above upper bound in
$1/|\log \vep|$ is sharp (See Proposition \ref{rrrem}).

\begin{Rem}
We can not expect to get an upper  bound  in $\frac{C}{|\log \vep|}$
for any $\vep$ with $C$ independent from the degree $r$ in the above
Corollary. See Remark \ref{upper}.
\end{Rem}

\subsection{Non $h$-expansive examples in dimension one}
\subsubsection{Homoclinic tangency} Let $2\leq r\leq +\infty$.
Let $f$ be a $C^r$ interval map and $\Lambda$ be a hyperbolic repeller. We say that $f|_\Lambda$ has a non degenerate
 homoclinic tangency  if there exists a critical point
  (local extremum) $c\in [0,1]$  such that $c\in W^u(\Lambda)$, $f^k(c)\in \Lambda$ for some $k>0$ and $c$
  is non degenerate for $f^k$, i.e. $(f^k)^{(l)}(c)\neq 0$ for some finite $l\leq r$.
   Here the unstable manifold $W^u(\Lambda)$ of $\Lambda$ is defined as the set of points $x\in [0,1]$,
    such that for any neighborhood $V$ of $\Lambda$, the point $x$ belongs to $f^n(V)$ for some positive integer $n$.
 This notion is similar to
   the notion of homoclinic tangency for diffeomorphisms in higher dimensions.
    However by \cite{Koz}  this picture is not persistent under $C^r$ perturbations contrarily to
     higher dimensions (Newhouse phenomenon).

\begin{main}\label{homone}
Let $f$ be a $C^2$ interval map with a non degenerate homoclinic
tangency. Then there exists $C=C(f)$ such that  for any $0<\vep<\frac{1}{2}$  we have
$$h_{\loc}(f,\vep)\geq \frac{C}{|\log\vep|}.$$
\end{main}

In the next statement we see with the example of the quadratic map
that the inequality (\ref{rer}), $\forall \vep>0, \
h(f)-h(f,\vep)\leq h_{\loc}(f,\vep)$, which holds for any continuous
dynamical system $(f,M)$, may be quite rough.

\begin{Prop}\label{rrrem}
The quadratic map $f_4$ given by $f_4(x)=4x(1-x)$ for all $x\in
[0,1]$ has a homoclinic tangency for the repulsing fixed point $0$.
In particular we have
$$h_{\loc}(f_4,\vep)\geq O\left(\frac{1}{|\log\vep|}\right),$$
but for any  $\alpha\in(0,1)$,
$$h(f_4)-h(f_4,\vep)\leq o\left(\frac{\vep^\alpha}{|\log\vep|}\right).$$
\end{Prop}

\begin{Rem}The  $2$-full shift is a finite to one extension of the quadratic map $f_4$. It is known that the tail entropy is invariant under such extension, but it is  false for the $\vep$-tail entropy. Indeed the $2$-full shift is expansive, thus $h$-expansive, but according to the previous proposition we have $h^*(f_4,\vep)\geq O\left(\frac{1}{|\log\vep|}\right)$.
\end{Rem}

By the already mentioned result of Kozlovscki, Shen and
van Strien \cite{Koz}  hyperbolic and thus $h$-expansive maps
 form an open and dense set in the $C^r$ topology for any finite $r$. But we do not know what is the Lebesgue typical rate for a one parameter family.

\begin{Que} What is the Lebesgue typical  rate of $h^*(f_a,\vep)$
(or $h_{\loc}(f_a,\vep)$) in the quadratic family $f_a(x)=ax(1-x)$?
\end{Que}

\subsection{Non $h$-expansive $C^r$ robust examples in Newhouse domains for diffeomorphisms in higher dimensions}

  For any  $r\geq 2$, the $C^r$ Newhouse domain is defined
by the closure of $C^r$ diffeomorphisms with homoclinic tangencies.
We prove every map  in an open dense subset of Newhouse domain is
not $h$-expansive and we give an explicit lower bound of the
$\vep$-local entropy:

\begin{main}\label{non EE}
Let $M$ be a compact  smooth Riemannian surface. Assume $f\in \Diff^r(M)$
$(r\geq2)$  has a hyperbolic basic set whose stable and unstable
manifolds are tangent at some point. Then for any $C^r$ neighborhood
$\mathcal{V}$ of $f$ in $\Diff^r(M)$, there exists a $C^r$ open set
$\mathcal{U}\subset \mathcal{V}$ and  a constant $C>0$ such that for
any $f\in \mathcal{U}$,
\begin{eqnarray}\label{large tail entropy }\limsup_{\vep\to
0}h_{\loc}(f,\vep)|\log \vep|> C.\end{eqnarray}Consequently,
everyone in $\mathcal{U}$ is not $h$-expansive.
 \end{main}

This theorem may be considered as a converse of  the result in
\cite{LVY}. Diffeomorphisms $C^1$ far from the set of
diffeomorphisms exhibiting a homoclinic tangency are robustly
$h$-expansive whereas in a $C^r$ $(r\geq 2)$ open dense subset
 of the set of diffeomorphisms with homoclinic tangencies  all systems are non $h$-expansive.

\begin{Rem}
The previous lower bound on the local entropy also holds on Newhouse
intervals in any one-parameter family which unfolds a quadratic homoclinic tangency generically,
 for example in the conservative H\'enon family $(x,y)\mapsto (y,-x+a-y^2)$.
\end{Rem}

 Since  polynomial maps are dense in the space of $C^r$ maps on any open bounded set of
$\mathbb{R}^m$
 (see also Proposition \ref{rrrem} and the above Remark for explicit examples),
  we have the following Corollary which contradicts Conjecture 6.2 of Yomdin \cite{Yom91} :

\begin{maincor}\label{noncor EE}
There exist non $h$-expansive polynomial maps  satisfying (\ref{large
tail entropy }).
 \end{maincor}

\begin{Rem}
Corollary \ref{noncor EE} shows that analytic maps may have
exponential dynamical  complexity in any scale. However, in the
setting of geometry, any $l$-dimensional analytic manifold $A\subset
\mathbb{R}^m$  always has no exponential complexity in any scale,
due to the property that for a constant $C(A)$,  for any cube
$Q_t^m$ of the size $t$ in $\mathbb{R}^m$ and for any affine $L:
\mathbb{R}^m\to \mathbb{R}^m$,
 $$\Vol(L(A)\cap Q_t^m)\leq C(A)t^l,$$
see Corollary 6.4 of \cite{Yom91}.
\end{Rem}

\begin{Que}
In the previous section we establish in dimension one that the rate of convergence of polynomials was in $1/|\log\vep|$. Does it hold true in higher dimensions?
\end{Que}

\subsection{$h$-expansiveness for endomorphisms on  homogenous space}
We observed in the previous section that polynomial maps are not $h$-expansive in general, however we will see now it is always the case of affine maps.

Let $G$ be a real Lie group and let $\Lambda$ be a discrete cocompact subgroup (usually called a uniform lattice).
 The quotient $G/\Lambda$  inherits from $G$ a structure of smooth manifold. It is well known that one may endow $G$ with a
biinvariant Riemannian  metric  as $G$ admits a uniform lattice. It induces a left invariant Riemannian metric on  $G/H$.  Endowed with this metric we call the quotient $G/\Lambda$  a compact  homogenous Riemanian manifold.


A map $\phi:G/\Lambda\rightarrow G/\Lambda$ is called an endomorphism of $G/\Lambda$ when $\phi$
is the map induced by an element $\underline{g}\in G$ and a morphism of group $\Phi:G\rightarrow G$  with $\Phi(\Lambda)\subset \Lambda$ as follows :
$$\phi(g\Lambda)=\underline{g}\Phi(g)\Lambda.$$

In this setting  Bowen \cite{Bow71} proved that the topological
entropy of $\phi$ is given by $\sum_i\log^+|\lambda_i|$ where
$\lambda_i$ are the eigenvalues of $d\Phi:T_eG\rightarrow T_eG$.
For endomorphisms on Lie groups, $h$-expansiveness has been
established by Bowen in \cite{Bow72b}. Here we show
$h$-expansiveness for all endomorphisms on homogenous manifolds.

\begin{main}\label{homogenous}
Let $G/\Lambda$ be a compact  homogenous Riemanian manifold.
 Then any endomorphism of $G/\Lambda$ is $h$-expansive.
\end{main}

\subsection{Arbitrarily slow convergence for $C^\infty$ maps and sharpness of Theorem \ref{gen}}
For general $C^\infty$ maps the convergence to zero of
the $\vep$-local entropy may be arbitrarily slow.

\begin{main}\label{thm2}Let $M$ be a compact smooth Riemannian manifold of dimension larger than one
 (resp. of dimension one). Let $f:M\rightarrow M$ be a $C^\infty$ diffeomorphism (resp. non invertible map)
 with an interval of homoclinic tangencies. Then,  for any function $a: (0,1)\to \mathbb{R}^+$ with $a(t)\to 0$ as $t\to 0$
 and for any $C^\infty$-neighborhood $\mathcal{U}$ of $f$, there
 exists a diffeomorphism (resp. non invertible map) $f_a\in \mathcal{U}$ on $M$ and $\zeta(f_a)>0$
 such that $$h_{\loc}(f_a,\vep)\geq a(\vep)\quad \text{for any}\quad \vep\in(0,\zeta(f_a)). $$
\end{main}

Moreover one can ensure this perturbation $f_a$ to be volume
preserving when this is the case of $f$.  Gonchenko, Turaev and
Shilnikov \cite{Turaev} have shown that volume preserving surface
diffeomorphisms with such an interval of  tangencies are $C^\infty$
dense in Newhouse domains. Therefore we get as a corollary:

\begin{maincor}\label{thm22}Let $M$ be a compact smooth Riemannian surface.
 For any function $a: (0,1)\to \mathbb{R}^+$ with $a(t)\to 0$ as $t\to 0$, there exists a $C^\infty$
 dense subset $ \cF_a$ of volume preserving diffeomorphisms in   $C^{\infty}$  Newhouse domain, such that we have for all $f\in \cF_a$
and for all $\vep$ small enough, $0\leq \vep <\zeta(f)$,
$$ h_{\loc}(f,\vep)\geq a(\vep). $$
\end{maincor}

We also prove the estimate obtained in Theorem \ref{gen} is sharp in the following sense.

\begin{main}\label{exa}
Let $M$ be a compact smooth manifold of dimension larger than one (resp. of dimension one). There exists a smooth metric on $M$ such that the following holds.

 For any  nondecreasing function $a: (0,1)\to \mathbb{R}^+$ with $1/a(e^{-.})$ concave
 on   $(0,+\infty)$,  $a(t)\to 0$ as $t\to 0$ and $a(t)\geq  t^{1/7}$ for all $t\in (0,1)$,
 there exists a logarithmic convex weight $\mathcal{M}=(M_k)_k$ satisfying  $ \frac{\log M_0}{G_\cM(3|\log\vep|)}\leq a(\vep)$ for all $\vep>0$,
 with the following property.

For any $\vep$ small enough, $0<\vep<\zeta(M,a)$, there exists
$f_\vep\in \Diff_V^{\mathcal{M}}(M)$ (resp. $f_\vep\in
C_V^{\mathcal{M}}(M)$) with
$$h^*(f_\vep,\vep)\geq a(\vep)\geq \frac{\log M_0}{G_\cM(3|\log\vep|)}.$$
\end{main}

\begin{Rem} The condition of concavity of $ 1/a(e^{-.}):(0,+\infty) \to \mathbb{R}^+$ is not very restrictive. Indeed for any nondecreasing bounded  function $a_0: (0,1)\to \mathbb{R}^+$ with $\lim_{\vep \to 0}a_0(\vep)=0$
there exists a  nondecreasing function $a_1: (0,1)\to \mathbb{R}^+$
with $ \lim_{\vep \to 0}a_1(\vep)=0$ satisfying this condition and
$a_1\geq a_0$. This follows easily from the fact that any
nondecreasing function $f:(0,+\infty)\to \mathbb{R}^+$ with $\lim_{x
\to 0} f(x)>0$ and $\lim_{x \to +\infty} f(x)=+\infty$  is larger
than a concave nondecreasing function $g:(0,+\infty)\to
\mathbb{R}^+$ satisfying also $\lim_{x \to 0} g(x)>0$ and $\lim_{x
\to +\infty} g(x)=+\infty$.
\end{Rem}

\subsection{Modulus of upper semicontinuity of the topological entropy}
 We state now, in the same spirit
of \cite{Yom91}, how our uniform estimates on the $\vep$-local entropy may be used to explicit a modulus of  continuity of the topological entropy  for the $C^0$ topology.

\begin{Prop}\label{mod} Let $f\in C^0(M)$ and let $G$ be  a subset of $C^0(M)$ such that
 $h^G_{\loc}(\vep):=\sup_{g\in G}h_{\loc}(g,\vep)$ goes to zero when $\vep$
goes to zero and $$M_0(G):=\sup_{g\in
G}\max(\|Dg\|_\infty,2)<\infty.$$ Let $p_\vep$ be the least integer
satisfying
 $\frac{1}{p_\vep}\log r_{p_\vep}(f, \vep/4)-h(f,\vep/4)\leq h^G_{\loc}(\vep)$.
  Then for any $\vep$ and for any $g\in G$ with $d(g, f):=\sup_{x\in M}d(gx,fx)\leq \vep$  we
  have
$$h(g)\leq h(f)+2h^G_{\loc}(N(\vep))$$
where $N(\vep)$ denotes the inverse function of $\vep
\mapsto \frac{\vep}{4} M_0(G)^{-p_\vep}$, i.e. $N(\vep)$ is the smallest positive real number such that $\frac{N(\vep)}{4}M_0(G)^{-p_{N(\vep)}}=\vep$.
\end{Prop}

Using this Proposition to study the continuity of entropy, the main
difficulty is left to estimate $p_\vep$ for a given $f\in C^0(M)$.
In Section \ref{modu} we prove Proposition \ref{mod} and apply it to
some (elementary) examples.

\section{Rate of convergence for ultradifferentiable maps}
In this section we devote to study the  tail  entropy and local volume growth for general
$C^{\infty}$ smooth maps beyond analytic maps. We are going to start
by improving the classical semi-algebraic theory used by Yomdin
\cite{Yom87-2}, Gromov \cite{Gr85} and Buzzi \cite{Buz97}.

\subsection{Buzzi estimates on the tail entropy via the algebraic lemma}

Following Yomdin's and Gromov's works to bound the local volume
growth, Buzzi \cite{Buz97} proved asymptotic $h$-expansiveness for
$C^\infty$ maps. As a first step, he  proved the following upper
bound of the tail entropy of some iterate of  a $C^r$ map $f$ (see
the proof of Theorem 2.2 in \cite{Buz97}).  We let $exp$ denote the
exponential map of the Riemanian manifold $M$. To simplify the
notation we write then $\|D^{k+1}\vep^{-1} f^p\vep\|_\infty$ for
$\|D^{k+1}\vep^{-1}exp_{f^px}^{-1}\circ f^p\circ
exp_x(\vep.)\|_\infty$ for all $0\leq k\leq r-1$.

\begin{Prop}\label{main}
Let $r>1$, $l\leq m\in \mathbb{N}$ and $p\in \mathbb{N}$. Let $f\in
C^r(M)$ with the dimension of $M$ equal to $m$ and $\vep>0$ such that  we have $\|D^{k+1}\vep^{-1}
f^p\vep\|_\infty\leq \max(\|Df^p\|,1)$ for all $1\leq k< r$. Let  $C_{r,l,m}$ as in
the Algebraic Lemma  and let $\tilde{C}_{r,l,m}=2^{l+2m}C_{r,l,m}$. Then

$$v_l^*(f^p,\vep)  \leq \frac{l}{r}\left(\log^+\|Df^p\|+2\log B_r\right)+\log \tilde{C}_{r,l,m}\,$$
and $$h^*(f^p,\vep)  \leq \frac{m}{r}\left(\log^+\|Df^p\|+2\log
B_r\right)+\log \tilde{C}_{r,m,m}\,$$ where $B_r$ is the $r^{th}$
Bell number.
\end{Prop}

We first  recall Faa di Bruno formula for the derivative of a
composition.

\begin{Lem}\label{Bruno}
$$(f\circ g)^{(k)}=\sum_{l=1}^kf^{(l)}\circ g \times B_{k}^l(g',g'',\cdots,g^{(k-l+1)})$$
with $B_{k}^l$  the so-called Bell polynomials given by
\begin{eqnarray*}&&B_{k}^l(X_1,\cdots,X_{k-l+1})\\[2mm]&=&\sum \frac{k!}{j_1!j_2!\cdots
j_{k-l+1}!}\left(\frac{X_1}{1!}\right)^{j_1}
\left(\frac{X_2}{2!}\right)^{j_2}\cdots\left(\frac{X_{k-l+1}}{(k-l+1)!}\right)^{j_{k-l+1}}\end{eqnarray*}
where the sum holds over all $j_1,j_2,\cdots,j_{k-l+1}\in \NN$ with $\sum
j_i=l$ and $\sum ij_i=k$.
\end{Lem}

The $r^{th}$ bell number  $B_r:=\sum_{l=1}^rB_{r}^l(1,\cdots,1)$
counts the class of all partitions of $\{1,\cdots,n\}$. It also
counts the class of all distributions of $n$  labeled balls among
$n$ indistinguishable urns. Therefore $B_r$ is less than the
cardinality of the class of distributions of $n$ labeled balls among
$n$ labeled urns, the latter class having $r^r$ members.

Proposition \ref{main} follows from the following lemma by
considering for all $n$ and for a fixed $x\in M$ the family of maps $(T_n)_{n\in \NN}$ given by
$$T_n=\vep^{-1}exp_{f^{p(n+1)}x}^{-1}\circ f^p\circ
exp_{f^{pn}x}(\vep\,\cdot\, ). $$ We refer to the original work of
Yomdin \cite{Yom87} and Buzzi (Proposition 3.3 of \cite{Buz97}) for
this step but we give a precise form of the estimation bound in
terms of Bell numbers.

\begin{Lem}\label{bu}
 Let $\sigma:[0,1]^l\rightarrow \mathbb{R}^m$ be a $C^r$ $l$-disk with
$\|\sigma\|_r<1$ and $(T_n:(-2,2)^m\rightarrow \mathbb{R}^m)_n$ be a
family of $C^r$ maps with $\|D^{k+1}T_n\|_\infty\leq A_n$ for all
$1\leq k \leq r-1$ for a sequence $(A_n)_n$ satisfying $A_n\geq \max(\|DT_n\|_\infty,1)$ for all $n$. Then for all $n$ there exists a family  $\mathcal{F}_n:=(\psi_n:[0,1]^l\rightarrow [0,1]^l)$ of continuous maps,  real
analytic on $(0,1)^l$,
such that with $T^n:=T_n\circ\cdots\circ T_1$ :

\begin{itemize}
\item $\|T^n\circ \sigma\circ \psi_n\|_r\leq 1$;\\
\item $\|D\left(T^k\circ \sigma \circ \psi_n\right)\|_\infty\leq 1$ for all $0\leq k\leq
n$;\\
\item $\bigcap_{k=0,1,\cdots,n}(T^k\circ \sigma)^{-1}((-1,1)^m)\subset \bigcup_{\psi_n\in
\mathcal{F}_n}\psi_n([0,1]^l)$;\\
\item $\sharp \mathcal{F}_{n+1}\leq  \tilde{C}_{r,l,m} \sharp \mathcal{F}_{n}\cdot
\left(A_nB_r^2\right)^{\frac{l}{r}}.$
\end{itemize}

\end{Lem}

\begin{proof}
We argue by induction on $n$. Assume the lemma holds for $n$ and let
us prove it for $n+1$. For all $\psi_n\in \mathcal{F}_n$ we have by
Faa di Bruno formula:

\begin{eqnarray*}
&&\|D^r(T_{n+1}\circ T^n\circ \sigma\circ \psi_n)\|_\infty \\[2mm]& \leq
& \sum_{k=1}^r \|D^kT_{n+1}\|B_r^k\left(\|D \left(T^n\circ \sigma\circ
\psi_n\right)\|_\infty,\cdots,\|D^{(r-k+1)} \left(T^n\circ
\sigma\circ \psi_n\right)\|_\infty\right).
\end{eqnarray*}
 and then by the induction hypothesis and the hypothesis on the higher derivatives of
 $T_{n+1}$:
\begin{eqnarray*}
\|D^r(T_{n+1}\circ T^n\circ \sigma\circ \psi_n)\|_\infty & \leq & \sum_{k=1}^r \|D^kT_{n+1}\|B_r^k(1,\cdots,1)\\
&\leq & A_{n+1}B_r.
\end{eqnarray*}
Therefore,  up to subdivide $[0,1]^l$ into
$\left(A_{n+1}B_r^2\right)^{\frac{l}{r}}$ subcubes and
to reparametrize them affinely from $[0,1]^l$,  we can assume

\begin{eqnarray}\label{one}
\|D^r\left(T_{n+1}\circ T^n\circ \sigma\circ
\psi_n\right)\|_{\infty} & \leq & 1/B_r.
\end{eqnarray}
Now if $P$ is the $r^{th}$ Lagrange polynomial at the center of
$[0,1]^l$ of $T_{n+1}\circ T^n\circ \sigma \circ \psi_n$ there
exists by Algebraic Lemma  a family of maps
$(\phi:[0,1]^l\rightarrow [0,1^l])$ satisfying
$$P^{-1}([-2,2]^m)=\bigcup_{i=1,\cdots,4^mC_{r,l,m}}\phi_i([0,1]^l).$$ In
particular as we have by Taylor formula $\|T_{n+1}\circ T^n\circ
\sigma \circ \psi_n-P\|_\infty\leq 1$, the maps
$\psi_{n+1}:=\psi_n\circ \phi$ satisfy
$$\bigcap_{k=0,1,\cdots,n+1}(T^k\circ \sigma)^{-1}((-1,1)^m)\subset
\bigcup_{\psi_{n+1}\in \mathcal{F}_n}\psi_{n+1}([0,1]^l).$$ Moreover
$\|P\circ \phi_i\|_r\leq 1$ and $\|\phi_i\|_r\leq 1$ for each $i$
and therefore by using again Faa di Bruno formula together with
(\ref{one}) we get:

\begin{eqnarray*}
\|T_{n+1}\circ T^n\circ \sigma \circ \psi_{n+1}\|_r & \leq & 1+ \|(P-T_{n+1}\circ T^n\circ \sigma \circ \psi_n)\circ \phi\|_r\\
& \leq  & 1+\sum_{k=1}^r\|D^k(P-T_{n+1}\circ T^n\circ \sigma \circ \psi_n)\|B_r^k(1,\cdots,1)\\
&\leq & 1+\sum_{k=1}^r\|D^r(T_{n+1}\circ T^n\circ \sigma \circ
\psi_n)\|B_r^k(1,\cdots,1)\leq 2.
 \end{eqnarray*}

Up to subdivide again $[0,1]^l$ into $2^l$ isometric subcubes and to reparametrize them affinely from $[0,1]^l$ we get $\|T_{n+1}\circ T^n\circ \sigma \circ \psi_{n+1}\|_r  \leq 1$.

Finally,   for all $0\leq k\leq n$ we have

\begin{eqnarray*}
\|D\left(T^k\circ \sigma \circ \psi_{n+1}\right)\|_\infty&=
&\|D\left(T^k\circ \sigma \circ \psi_n\circ
\phi\right)\|_\infty\\[2mm]
&\leq &  \|D\left(T^k\circ \sigma \circ
\psi_n\right)\|_\infty\|D\phi\|_\infty\leq 1.
\end{eqnarray*}
This proves the statement for $n+1$ and concludes the proof by
induction of Lemma \ref{bu}.

\end{proof}

Then we use the following lemma which relies the $\vep$-local volume
growth  and the $\vep$-tail entropy of $f$ with these of its
iterates $f^p$ to kill the constant term $\frac{2l}{r}\log B_r+\log
\tilde{C}_{r,l,m}$ in Proposition \ref{main}. It follows from two
facts. First the $(\varepsilon,np)$-dynamical  ball for $f$ is
contained in the $(\varepsilon, n)$-dynamical  ball for $f^p$ with
the same center. Secondly the growth of any $l$-disk under $f^k$
with $0\leq k\leq p$ is uniformly bounded by $\max(1,\|Df\|)^{pl}$
and  for any scale $\delta$ there exists a smaller scale $\delta'$
such that a $(np,\delta')$ spanning set for $f^p$ is $(\delta,n)$
spanning  for $f$.

\begin{Lem}\label{htailpower}
Let  $f\in C^1(M)$, and $g\in C^0(M)$, $\vep>0$ and $p\neq 0$ be an
integer. For any integer $l$ less than or equal to the dimension of
$M$, we have
\begin{eqnarray*}v_l^*(f,\vep)&\leq & v_l^*(f^p,\vep)/p\,;\\[2mm]
h_{\loc}(g,\varepsilon) &\leq&  h_{\loc}(g^p,\varepsilon)/\,p;\\[2mm]
h^*(g,\varepsilon) &\leq&  h^*(g^p,\varepsilon)/p\,.
\end{eqnarray*}
\end{Lem}

By taking the limit when $\vep$ goes to zero we have $v_l^*(f^p)\leq
v_l^*(f)/p$, $h_{\loc}(f) \leq h_{\loc}(f^p)/p$ and $h^*(f) \leq
h^*(f^p)/p$. The equalities $v_l^*(f^p)=v_l^*(f)/p$, $h_{\loc}(f) =
h_{\loc}(f^p)/p$ and $h^*(f) = h^*(f^p)/p$ hold also true but are not
used here.

\subsection{Rate of convergence of the tail entropy for ultradifferentiable maps}

We will make an adapted choice of $p$ and $r$ together to give a
precise rate of convergence of $V$-ultradifferentiable maps.\\

\noindent{\it Proof of Theorem \ref{gen}.\quad} We prove Theorem \ref{gen} for the $\vep$-local $l$-volume growth for some given $0\leq l\leq m$. The proof is analogous for the $\vep$-tail entropy and is left to the reader.

Now we fix a logarithmic weight $(M_k)_k$ and we consider $f\in C_V^\cM(M)$. Let $1>\gamma>0$. We
choose $r$ and then $p$ such that

\begin{itemize}\label{choice}
\item $r=\lceil \frac{l}{\gamma}\rceil$;\\
\item $ p = \lceil\frac{r\log (\tilde{C}_{r,l,m}r^{2l})}{Dl\log M_0}\rceil. $
\end{itemize}
In particular  we have
\begin{itemize}\label{choice}
\item $\frac{l}{r}\leq \gamma$;\\
\item $\frac{\log (\tilde{C}_{r,l,m}r^{2l})}{ p} \leq \frac{2D l\log M_0}{r}\leq 2D\gamma\log M_0.$
\end{itemize}

Then we fix $\vep$ so that the assumptions on the derivatives of
$f^p$ of Lemma \ref{bu} is checked with $A_n=M_0^p$ for all $n$, that is for all $1\leq
k< r$ :
\begin{eqnarray*}
\|D^{(k+1)}\vep^{-1}f^p\vep\|&\leq & M_0^p.
\end{eqnarray*}

Note that we need $\vep$  also satisfy
$$\vep M_0^p<R_{inj}.$$
\begin{Lem}
With the previous notations, we have for $k\geq 1$

$$\|D^{(k+1)}f^p\|\leq (kp)^kM_0^{(k+1)(p-1)}\max_{k_i\geq 1, \sum_i k_i=k}\,\,\prod_i M_{k_i}.$$
\end{Lem}

\begin{proof}
Let $k\geq 1$. Clearly  $D^{(k+1)}f^p$ is a polynomial in
$D^{(n+1)}f\circ f^m$ with $0\leq m \leq p-1$ and $0\leq n\leq k$.
By an easy induction the total degree of this polynomial is
$(k+1)(p-1)+1$, the degree of the variables involving the first
derivative of $f$ is at most $(k+1)(p-1)$ and  the number of
monomials does not exceed $k!p^k$. Also if we denote $l_n$ the
degree in the derivative of order $n+1$ we have $\sum_nnl_n=k$.
\end{proof}

We continue the proof of Theorem \ref{gen}.  It follows from the
logarithmic convexity  of the weight  $(M_k)_k$ that
$$\max_{k_i\geq1, \sum_i k_i=k}\,\,\prod_i M_{k_i}\leq M_0^k M_k.$$
According to the above lemma  we get then :
\begin{eqnarray*}
\|D^{(k+1)}\vep^{-1}f^p\vep\|&\leq & \vep^k(kp)^kM_0^{(k+1)p} (M_k/M_0).
\end{eqnarray*}
Therefore as $(\left(M_k/M_0\right)^{\frac{1}{k}})_k$ is nondecreasing,
we may choose
$$\vep=\frac{1}{rp}M_0^{-p} \min\left(R_{inj},\left(M_r/M_0\right)^{-\frac{1}{r}}\right).$$
Apply now Proposition \ref{main} with the previous datas. We get  :
\begin{eqnarray*}\label{choic}
\frac{1}{p}v_l^*(f^p, \vep) & \leq &
\frac{l}{rp}\log\|Df^p\|+\frac{\log(
\tilde{C}_{r,l,m}B_r^{\frac{2l}{r}})}{p}\\[2mm]
& \leq & \frac{l}{r}\log M_0+\frac{\log (\tilde{C}_{r,l,m}r^{2l})}{p}\\[2mm]
& \leq & (2D+1)\gamma \log M_0.
\end{eqnarray*}
According to Lemma \ref{htailpower} we have the following upper
bound on the local $l$-volume growth of $f$,
$$v_l^*(f, \vep) \leq (2D+1)\gamma \log M_0. $$

We explicit now the function $\vep=\psi(\gamma)$. In fact we give a
lower bound  $\varphi$ of $\psi$ which is increasing. Then we
inverse $\varphi$ to get $\gamma\leq \varphi^{-1}(\vep)$. We have :
\begin{eqnarray*}
\vep & = & \frac{1}{rp}M_0^{-p}\min\left(R_{inj},\left(M_r/M_0\right)^{-\frac{1}{r}}\right);
\\[2mm]
-\log \vep & = &\log (pr)+p\log M_{0}+ \max\left(\log^+(1/R_{inj}),\frac{\log (M_r/M_0)}{r}\right).\\
\end{eqnarray*}
Now we have by $(l,m,D)$-admissibility of $\mathcal{M}$
\begin{eqnarray*}
\log (pr)+p\log M_{0}&\leq &\log r+ (1+\log M_0)p,\\
&\leq & \log r +(1+\log M_0)\frac{r\log \left(\tilde{C}_{r,l,m}r^{2l}\right)}{Dl\log
M_0}\\
&\leq & \log r +\frac{2r\log \left(\tilde{C}_{r,l,m}r^{2l}\right)}{Dl},\\
&\leq & \frac{\log (M_r/M_0)}{r}.
\end{eqnarray*}
It follows that
\begin{eqnarray*}
-\log \vep&\leq & 2\max\left(\log^+(1/R_{inj}),\frac{\log (M_r/M_0)}{r}\right),
\end{eqnarray*}
that is for all $0<\vep<\min\left(1,R_{inj}^2\right)$ we have
\begin{eqnarray*}
-\log \vep&\leq & \frac{2\log (M_r/M_0)}{r}.
\end{eqnarray*}
Therefore by the definition of $G_l$ and then by replacing $r$ by
its expression in terms of $\gamma$
\begin{eqnarray*}
r &\geq & G_l\left(|\log \vep|/2\right);\\
\gamma & \leq & \frac{2l}{ G_l(|\log \vep|/2)}.
\end{eqnarray*}

\hfill $\Box$

\begin{Rem}
We gave here estimates of the rate of convergence of the
$\vep$-local entropy of $f$ through $\vep$-local volume growth. But
by using the same method we can deal directly with the measure
quantity, $h_{\loc}(\mu,\vep)$, and determine the rate of
convergence of $\left(h_{\loc}(\mu,\vep)\right)_\vep$ in terms of
the maximal positive Lyapunov exponent of $\mu$ instead of $\log
M_0\geq \log^+ \|Df\|$.
\end{Rem}

\subsection{Estimate of $C_{r,1,m}$}\label{quantgr}
In this section we will give an estimate of the algebraic constant
in dimension $1$ :

\begin{Prop}\label{grone}
There exists a constant $C$ such that for all $r$,
$$C_{r,1,m}\leq C \, m^3r^8. $$
\end{Prop}

\begin{Rem}In higher dimensional cases the proof of Gromov's algebraic lemma is
more complicated  and we do not plan to discuss this case in the
present paper. In \cite{paperinpreparation} the author proves that
for any dimension $m$ there exists a constant $A_m$  depending only
on $m$ such that $C_{r,m}\leq r^{A_mr^m}$.  With this estimate, by
applying Theorem \ref{gen} we can get  that for any analytic map $f$
on a compact smooth manifold $M$ of dimension $m$, it holds that
$$h^*(f,\vep)\leq B\left(\frac{\log|\log \vep|}{|\log
\vep|}\right)^{\frac{1}{m+1}}$$  for some constant $B=B(f)$ depending
 on $f$.

\end{Rem}

>From the point of view of  Proposition \ref{grone}, it seems
reasonable to ask the following question concerning the polynomial
growth of $C_{r,l,m}$ in $r$ for any dimension:

\begin{Que}
For any $m\in \NN$, do there exist constants $A_{m,l}, B_{m,l}$ such
that for all $r\in \NN$, $$C_{r,l,m}\leq A_{m,l} \,
r^{B_{m,l}} \ ?
$$
\end{Que}

 We will not directly adopt the proof
of Gromov \cite{Gr85} but give here a new proof of Algebraic Lemma
in the one dimensional case. In fact, by following straightforwardly
Gromov's work we  only manage to get the following super-exponential
growth upper bound, $C_{r,1,m}\leq Cm^310^{r^2}$. \\

\noindent{\it Proof of Proposition \ref{grone}}.\quad  Let $(P_1,...,P_m)\in
\mathcal{R}[X]^m$ be a finite family of polynomials  of degree less than or equal to $r$.\\

\emph{ First step : $\|P_j\circ\phi\|_1\leq 1$.} To bound the first
derivative, we consider one connected component of the following set

$$[0,1]\setminus \bigcup_{i,j}\{P_i=0,P_i=1,|P_i'|=|P_j'|, |P_i'|=1\}.$$

 Observe there are at most $4rm+2rm^2\leq 6rm^2$  components. The set
$\bigcap_jP_j^{-1}([0,1])$ is the closure of the union of some of these intervals. On such an interval $I$   we
have $P_j(I)\subset [0,1]$ for any $j=1,...,m$. Moreover there exists $i$ such that    $|P_i'(x)|=\max_{j}|P_j'(x)|$ for all $x\in I$ and we have either $|P_i'(x)|\leq 1 $ for all $x\in I$ or $|P'_i(x)|\geq 1$ for all $x\in I$. In the first case we just reparametrize $I:=[a,b]$ from $[0,1]$ by an affine contraction
$$\phi_I(t)=a+t(b-a)$$ while in the second  case of  we consider the
inverse of $P_i$ $$\phi_I(t):=P_i^{-1}(P_i(a)+t(P_i(b)-P_i(a))).$$ One easily
checks that $\|\phi\|_1\leq 1,\quad \|P_{j}\circ \phi\|_1\leq 1$ for any $j$. \\

\emph{ Second step : $(P_j\circ\phi)^{(k)}$ have constant sign for $k=2,...,r+1$.}
 We subdivide $[0,1]$ into subintervals
 where the derivatives  $\left(P_j\circ\phi_I\right)^{(k)}$  for $k=2,\cdots,r+1$ and for $j=1,...,m$
  have constant sign and therefore where $\left(P_j\circ\phi_I\right)^{(k)}$   for $k=1,\cdots,r$ are monotone. It is enough to consider one connected component
  of the  set  $$[0,1]\setminus \{(P_j\circ \phi_I)''=0;\cdots;P_j\circ \phi_I^{(r)}=0\}.$$
   When $\phi_I$ is just a linear contraction, $\phi_I(t)=a+t(b-a)$ for all $t\in [0,1]$, there are at most $$1+\sum_j deg(P_j'')+ \cdots+deg(P_j^{(r)})\leq mr^2$$ such components. As before we reparametrize them from $[0,1]$ by affine contraction $t\mapsto c+t(d-c)$.

 For the second case, $\phi_I(t):=P_i^{-1}(P_i(a)+t(P_i(b)-P_i(a)))$  for all $t\in [0,1]$, we use the following lemma.

\begin{Lem}
Let $ k\geq 1$. Then there exists a polynomial $R\in
\mathbb{R}[X_1,...,X_k]$ of total degree $k-1$ such that
\begin{eqnarray*}&&\left(P_i^{-1}\left[P_i(a)+.(P_i(b)-P_i(a))\right]\right)^{(k)}\\[2mm] &=& \frac{R(P_i'\circ
P_i^{-1}(a+.(b-a)),\cdots,P_i^{(k)}\circ
P_i^{-1}(P_i(a)+.(P_i(b)-P_i(a))}{\left(P_i'\circ
P_i^{-1}(P_i(a)+.(P_i(b)-P_i(a))\right)^{2k-1}}.\end{eqnarray*}
\end{Lem}

In particular the numerator in the above lemma is a polynomial of
degree at most $k(r-1)$ in $P_i^{-1}(P_i(a)+.(P_i(b)-P_i(a)))$. By Faa di Bruno formula it follows that
$\left(P_j\circ P_i^{-1}\left[P(a)+.(P(b)-P(a))\right]\right)^{(k)}$ may be written as a rational function with a polynomial numerator of degree at most $(k+1)(r-1)$ in $P_i^{-1}(P_i(a)+.(P_i(b)-P_i(a)))$ and therefore
$\left(P_j\circ P_i^{-1}(P_i(a)+.(P_i(b)-P_i(a)))\right)^{(k)}$ has at most $(k+1)(r-1)$
zeroes in $[0,1]$. Thus up to subdivide $[a,b]$ into at most $mr^3$
intervals one can assume
$\left(P_j\circ P_i^{-1}(P_i(a)+.(P_i(b)-P_i(a)))\right)^{(k)}$  for $k=1,\cdots,r+1$
have constant sign.  We reparametrize all
these subintervals affinely from $[0,1]$. Note that after this first
step we get at most $Cm^3r^4$ reparametrizations.\\

\emph{ Third step : $\|P_j\circ\phi\|_r\leq 1$.}
We let $H:[0,1]\rightarrow \mathbb{R}$ be a $C^{r+1}$ function such that the derivatives $(H^{(k)})_{k=2,...,r+1}$ have constant signs and such that $\|H\|_1\leq 1$. We will show that $\|(H\circ Q_r)^{(k)}\|_\infty\leq Cr^{4k}$ for the reparametrization $Q_r:[0,1]\rightarrow [0,1]$ defined in the following lemma. Then to conclude the proof of Proposition \ref{grone} one apply this result to the maps $H=P_j\circ\phi$ where $\phi$ are the reparametrizations obtained at the end of the second step.

\begin{Lem}
There exists a unique polynomial $Q_r$ of degree $2r-1$ such that
$Q(0)=0$, $Q(1)=1$ and
 $Q^{(k)}(0)=Q^{(k)}(1)=0$ for $k=1,\cdots,r-1$.

Moreover $Q_r$ satisfies the following properties :
\begin{itemize}
\item $Q_r$ satisfies the functional equation $1-Q(1-X)=Q(X)$;\\
\item $Q_r$ is an homeomorphism from $[0,1]$ onto itself;\\
\item $Q'_r(X)=b_rX^{r-1}(1-X)^{r-1}$ where $1/b_r=\beta(r,r)$ where $\beta$ is the usual $\beta$
function;\\
\item $Q_r(x) \geq b_rx^r(1-x)^r$.
\end{itemize}
\end{Lem}
\begin{proof}
We only prove the last item (the other statements are easy to
check). By the third item we have

$$Q_r(x)=\int_0^x b_rt^{r-1}(1-t)^{r-1}dt. $$
Then by considering the change of variable $t=ux$ we get
\begin{eqnarray*}
Q_r(x)&=&x^r\int_0^1 b_ru^{r-1}(1-ux)^{r-1}du\\[2mm]
&\geq & x^r\int_0^1 b_ru^{r-1}\left((1-u)(1-x)\right)^{r-1}du\\[2mm]
&\geq & x^r(1-x)^r b_r.
\end{eqnarray*}\end{proof}

Fix $1\leq k\leq r$, $1\leq l\leq k$ and
$\underline{j}:=(j_1,j_2,\cdots,j_{k-l+1})$ as in Faa di bruno
formula (Lemma \ref{Bruno}) and consider the polynomial
$$T_{l,\underline{j}}:=\left(\frac{Q_r'}{1!}\right)^{j_1}\left(\frac{Q_r''}{2!}\right)^{j_2}\cdots
\left(\frac{Q_r^{(k-l+1)}}{(k-l+1)!}\right)^{j_{k-l+1}}.$$
 We let
$S$ be the polynomial $S(X):=X(1-X)$. Recall $Q'_r=b_rS^r$ and
$\|S\|_\infty=1/4$.

\begin{Lem}\label{Qderivative}
Let $0\leq i\leq k-1$. Then there exists a polynomial $R_i$ with
$\|R_i\|_\infty\leq (r/2)^ii!$ such that
$$Q_r^{(i+1)}=b_rS^{r-i}R_i \,. $$
In particular as $b_r=\frac{(2r-1)!}{(r-1)!^2}\leq C\sqrt{r}2^{2r}$,
we have

$$Q_r^{(i+1)}\leq C\sqrt{r}2^{2i}\|R_i\|_\infty\leq C\sqrt{r}(2r)^ii!\leq Cr^{2k}. $$
\end{Lem}

\begin{proof}
We argue by induction on $i$. Observe $R_0=1$. The polynomials $R_i$
satisfies the following property :
$$R_{i+1}=(r-i)S'R_i+SR_i'\,.$$
 In
particular the degree of $R_i$ is equal to $i$. Now by Markov
inequality, $$\|R'_{i}\|_{\infty} \leq
2i^2\|R_{i}\|_{\infty}$$ (the norm $\|\cdot \|_\infty$ is the
classical supremum norm over $[0,1]$) and therefore
\begin{eqnarray*} \|R_{i+1}\|_\infty & \leq &
(r-i)\|S'\|_\infty\|R_{i}\|_\infty+\|S\|_\infty\|R_i'\|_\infty\\[2mm]
&\leq & \|R_{i}\|_\infty(r-i +i^2/2)\\[2mm]
&\leq & \|R_{i}\|_\infty r(i+1)/2\\[2mm]
&\leq & (r/2)^{i+1}(i+1)!.
\end{eqnarray*}

\end{proof}

Let us  bound from above the supremum norm of $H^{(l)}\circ Q_r
\times T_{l,\underline{j}}$ over $[0,1]$.

\begin{eqnarray*} \|H^{(l)}\circ Q_r \times
T_{l,\underline{j}}\|_\infty &\leq &
 b_r^l\|H^{(l)}\circ Q_r \times S^{(r+1)l-k}\|_\infty \times
 \prod_{i=1}^{k-l+1}\left(\|R_{i-1}\|_\infty/i!\right)^{j_i}\\[2mm]
&\leq & b_r^l(r/2)^{k}\|H^{(l)}\circ Q_r \times
S^{(r+1)l-k}\|_\infty\\[2mm]
&\leq & b_r^l(r/2)^{k}\|H^{(l)}\circ Q_r \times
S^{r(l-1)}\|_\infty\\[2mm]
&\leq & (r/2)^{k}\|H^{(l)}\circ Q_r \times
Q_r^{l_1}Q_r(1-.)^{l_2}\|_\infty.
\end{eqnarray*}
where $l_1$ (resp. $l_2$) is the number of $2\leq m\leq l$ such that
$|P^{(m)}|$ is non-increasing (resp. nondecreasing).

Consider finally the term $\|H^{(l)}\circ Q_r \times
Q_r^{l_1}Q_r(1-.)^{l_2}\|_\infty$. Assume first that $|H^{(l)}|$ is
non-increasing on $[0,1]$. Then we have for all $2\leq m\leq l$

$$|H^{(l)}(1-Q_r(1-x)m/l)|\leq |H^{(l)}(1-Q_r(1-x))|=|H^{(l)}\left(Q_r(x)\right)|$$
and

\begin{eqnarray*}
|H^{(l)}(Q_r(x)m/l)\times Q_r(x)/l|&\leq &
\left|\int_{Q_r(x)(m-1)/l}^{Q_r(x)m/l}H^{(l)}(t)dt\right|\\[2mm]
&\leq& \max(H^{(l-1)}(Q_r(x)m/l),H^{(l-1)}(Q_r(x)(m-1)/l).
\end{eqnarray*}


When $|H^{(l)}|$ is nondecreasing on $[0,1]$ we get symmetrically

$$|H^{(l)}(Q_r(x)m/l)|\leq |H^{(l)}(Q_r(x))|=|H^{(l)}(Q_r(1-x)|$$
and
\begin{eqnarray*}
&&\|H^{(l)}(1-Q_r(1-x)m/l) \times Q_r(1-x)/l \|_\infty\\[2mm] &\leq &
\max\left(H^{(l-1)}(1-Q_r(1-x)m/l),H^{(l-1)}(1-Q_r(x)(m-1)/l\right).
\end{eqnarray*}
By an easy induction one obtains $$\|H^{(l)}\circ Q_r \times
Q_r^{l_1}Q_r(1-.)^{l_2}\|_\infty\leq l^{l-1}\|H'\|_\infty\leq l^l,
$$ and then $$\|H^{(l)}\circ Q_r \times
T_{l,\underline{j}}\|_\infty \leq (\frac{r}{2})^{k}\cdot l^{l}.$$
Finally by the identity $B_{k,l}(1!,\cdots,(l-k+1)!)=
C_k^lC_{k-1}^{l-1}(k-l)!$ we get
\begin{eqnarray*}
\|(H\circ Q_r)^{(k)}\|_\infty&\leq&\sum_{l=1}^k  (\frac{r}{2})^{k}\cdot l^{l}\cdot B_{k,l}(1!,\cdots,(l-k+1)!)\\
&\leq & \sum_{l=1}^k   r^{k}\cdot l^{l} \cdot C_k^lC_{k-1}^{l-1}(k-l)!\\
&\leq &\sum_{l=1}^k  r^kl^{l} k^{2k}/l!\\
&\leq & Cr^{4k}.
\end{eqnarray*}
We have also for $1\leq k\leq r$ by Lemma \ref{Qderivative},
$$\|Q_r^{(k)}\|\leq Cr^{2k}\leq Cr^{4k}.$$
To conclude the proof of Lemma \ref{ygpol},  subdivide the unit
interval into at most $[Cr^4]+1$ intervals $I$ of length  $1/Cr^4$
and let $\psi_I$ be the affine reparametrizations from $[0,1]$ of
$I$. One easily checks that $\|Q_r\circ\psi_I\|_r,\|P\circ Q_r\circ
\psi_I\|_r\leq 1$ so that the family of reparametrization $Q_r\circ
\psi_I$ satisfy the conclusions of Lemma \ref{ygpol}.

 \hfill $\Box$

\subsection{Surface diffeomorphisms: Proof of Corollary \ref{surfh}}

  With the assumptions in Corollary \ref{surfh}, $\cM=(k^{k^2})_k$, we have  for all integers $k\neq 0$
$$ \frac{\log M_k}{k}=k\log k,$$
 and thus $G_\cM$ may be bounded from above as
follows :
\begin{eqnarray*}
l&=& \frac{\log M_k}{k} \\[2mm]
&= & k\log k; \\[2mm]
\log l& \geq& \log k,
\end{eqnarray*}
and then the function $G_\cM$  satisfies :
\begin{eqnarray*}
G_\cM(l)&=&k\\[2mm]
&=  & \frac{ l}{\log k}\\[2mm]
&\geq & \frac{ l}{\log l }.
\end{eqnarray*}
Thus, with $C=C(f,\cM)$ the constant in Corollary \ref{oups},

\begin{eqnarray*}
h_{\loc}(f,\vep)&\leq &v_1^*(f,2\vep)\\[2mm]
&\leq &\frac{C}{ G_\cM(|\log (2C\vep)|/2)}\\ [2mm]
&\leq &C\frac {2\log\left(|\log(2C\vep)|/2\right)}{|\log(2C(\vep)|}\\[2mm]
&\leq &\frac {\widetilde{C}(f,\cM)\log|\log\vep|}{|\log\vep|}
\end{eqnarray*}  for some
constant $\widetilde{C}(f, \cM)$.

\begin{Rem}
Corollary \ref{surfh} holds also true for local surface
diffeomorphisms. In fact one has again in this  case
$h_{\loc}(f,\vep)\leq v_1^*(f,2\vep)$ for any $\vep$ small enough.
Indeed it was proved for local diffeomorphisms in \cite{burex}
(Theorem 5) that there exists $\vep>0$ such that any invariant
measure $\mu$ with $h_{\loc}(\mu,\vep)>0$
 has at least one negative Lyapunov exponent.
\end{Rem}

\section{The case of one dimensional multimodal maps}

We prove in this section all the results related to one dimensional dynamics :  Theorem \ref{wmulti}, Theorem \ref{quasionedim} and Theorem \ref{homone}. We will make use
of the following lemma of analysis.

\begin{Lem}\label{easy}
Let $k\geq 1$ and $f$ be a $C^{k+1}$ map of the interval $I$.
 If the derivative $f'$ of $f$ vanish at $x_1<x_2<\cdots<x_k$  then for any $x\in I$ we have
$$|f'(x)|\leq \|f^{(k+1)}\|_\infty |I|^{k}.$$
\end{Lem}
\begin{proof}By the assumptions, for any $1\leq l \leq k$, there exists $y_l\in [x_1,x_k]$ such that
$f^{(l)}(y_l)=0$. Therefore, for any $x\in I$,
\begin{eqnarray*}|f^{(k)}(x)|&=&|\int_{y_k}^{x}f^{(k+1)}(z)dz|\leq  |I|\|f^{(k+1)}\|_\infty;\\
|f^{(k-1)}(x)|&=&|\int_{y_{k-1}}^{x}f^{(k)}(z)dz|\leq  |I|^2\|f^{(k+1)}\|_\infty;\\
&\vdots&\\
|f'(x)|&=&|\int_{y_{1}}^{x}f^{(2)}(z)dz|\leq
|I|^{k}\|f^{(k+1)}\|_\infty.
 \end{eqnarray*}
\end{proof}

\noindent{\it Proof of Theorem \ref{quasionedim}}.\quad Let $f$ be a
$C^l$ $l$-multimodal map of the unit interval. By Proposition 2.5
\footnote{In \cite{LVY} the authors consider $h\left(f,
B_{+/-\infty}(f,x,\vep)\right)$ with
$B_{+/-\infty}(f,x,\vep)=\bigcap_{n\in
\mathbb{Z}}f^{-n}B(f^nx,\vep)$ for an homeomorphism $f$, but the
proof applies also in the noninvertible case with
$h\left(f,B_\infty(f,x,\vep)\right)$.}
 of \cite{LVY} it is enough to prove $ h\left(f, B_\infty(f,x,\vep)\right)\leq  \frac{\log^+ \|f\|_{l}}{|\log \vep|}$ for
  $\mu$ almost every $x$ of any invariant ergodic measure $\mu$. Let $\mu$ be such a measure.

  Fix $x\in [0,1]$ and $\vep>0$. Let $n\in\NN$ and $\vep>\delta>0$.
 It is easily seen that the maximal cardinality of an $(n,\delta)$ separated set in $B_n(f,x,\vep)$ is not more  than the
$n/\delta$ time the number of  monotonic  branches of $f^n$ intersecting
$B_n(f,x,\vep)$ (see for example \cite{DM09}). But the number of such
$f^n$-monotonic  branches is
 less than $\prod_{k=0}^{n-1}M_{f^kx,\vep}$ where $M_{y,\vep}\leq l$  is the
  number of $f$-monotonic branches in the $\vep$-ball at $y\in [0,1]$.
 Therefore we have for all  $n\in \NN$ and for all $0<\delta<\vep$
$$r_n(f,B_n(f,x,\vep),\delta)\leq \frac{n}{\delta}  \prod_{k=0}^{n-1}M_{f^kx,\vep}.$$
By the ergodic theorem, for $\mu$-almost  every $x$, the sequence  $\left(\frac{1}{n}   \sum_{k=0}^{n-1}
\log M_{f^kx,\vep}\right)_n$ converges to $M_{\mu,\vep}:=\int \log M_{y,\vep}d\mu(y)$ so that
 \begin{eqnarray}\label{upper bound}h(f,B_\infty(f,x,\vep))& \leq &
\lim_n\frac{1}{n}   \sum_{k=0}^{n-1}
\log M_{f^kx,\vep}; \nonumber\\
&\leq & M_{\mu,\vep} .\end{eqnarray}

Let
$t_n=\max(1,\max_{m=0,...,n-1}\prod_{k=0}^{m}\|f'|_{B(f^kx,\vep)}\|)$
and let $E\subset B_n(f,x,\vep)$ be a $\delta/t_n$ covering set of
the dynamical ball $B_n(f,x,\vep)$. Then $E$ is also
   an $(n,\delta)$-spanning
set of $B_n(f,x,\vep)$. Indeed for any  $y\in B_n(f,x,\vep)$ there
exists $z\in E$ with $d(y,z)<\delta/t_n\leq\delta$. Noting that
$B(x,\vep)$ is connected, so $d(fy,fz)\leq
d(y,z)\|f'|_{B(x,\vep)}\|\leq \delta$, $d(f^2y,f^2z)\leq
d(fy,fz)\|f'|_{B(fx,\vep)}\|\leq
d(y,z)\|f'|_{B(x,\vep)}\|\|f'|_{B(fx,\vep)}\|\leq \delta$, and
therefore by induction for any $1\leq m<n$ we have
$$d(f^{m}(y), f^{m}(z))\leq d(y,z)
\prod_{k=0}^{m-1}\|f'|_{B(f^kx,\vep)}\|\leq\delta.$$ We can take
$\sharp E\leq t_n/\delta$. It follows that
\begin{eqnarray*}
r_n(f,B_n(f,x,\vep),\delta)  \leq \sharp E\leq
\frac{1}{\delta}\max\left(1,\max_{m=0,...,n-1}\prod_{k=0}^{m}\|f'|_{B(f^kx,\vep)}\|\right).
\end{eqnarray*}

  By the  ergodic theorem $\left(\frac{1}{n}\left(\sum_{k=0}^{n-1}\log\|f'|_{B(f^kx,\vep)}\|\right)\right)_n$
 converges  for $\mu$-almost every $x$. Then for such $x$ one easily sees  that  if this limit to be
positive then $\left(\frac{1}{n}\max_{0\leq
m<n}\left(\sum_{k=0}^{m}\log\|f'|_{B(f^kx,\vep)}\|\right)\right)_n$
converges to the same limit; if this limit to be nonpositive then
$\left(\frac{1}{n}\max_{0\leq
m<n}\left(\sum_{k=0}^{m}\log\|f'|_{B(f^kx,\vep)}\|\right)\right)_n$
converges to 0.   We may assume this limit to be positive.

 Observe now that $f'$ has at least $M_{y,\vep}-1$ zeroes in $B(y,\vep)$. By Lemma
\ref{easy}, we get $$\|f'|_{B(y,\vep)}\|_\infty\leq
\|f^{(M_{y,\vep})}\|_\infty\,\vep^{M_{y,\vep}-1},$$
and then
\begin{eqnarray}\label{dm}
\prod_{k=0}^{n-1}\|f'|_{B(f^kx,\vep)}\|&\leq & \max\prod_{k=0}^{n-1}
\|f^{(M_{f^kx,\vep})}\|\vep^{M_{f^kx,\vep}-1};\\[2mm]
&\leq &
\vep^{\sum_{k=0}^{n-1}M_{f^kx,\vep}}\vep^{-n}
\|f\|_l^{n}.\nonumber
\end{eqnarray}
By  geometric-arithmetic mean inequality we get
$$\sum_k^{n-1}M_{f^kx,\vep}\geq n\left(\prod_{k=0}^{n-1}M_{f^kx,\vep}\right)^{1/n}.$$
Therefore,
 \begin{eqnarray*}
\sum_{k=0}^{n-1}\log\|f'|_{B(f^kx,\vep)}\|
& \leq & n
\left(\left(\prod_{k=0}^{n-1}M_{f^kx,\vep}\right)^{1/n}-1\right)\log\vep+n\log^+
\|f\|_{l};\\
h(f,B_\infty(f,x,\vep)) &\leq & \max\left(\limsup_n\frac{1}{n}\max_{0\leq m<n}\left(\sum_{k=0}^{m}\log\|f'|_{B(f^kx,\vep)}\|\right),\,\,0\right)\\
&= & \max\left(\lim_n\frac{1}{n}\left(\sum_{k=0}^{n-1}\log\|f'|_{B(f^kx,\vep)}\|\right),0\right)\\
&\leq & \max\left( (e^{M_{\mu,\vep}}-1)\log\vep+\log
^+\|f\|_{l},\,\, 0\right).
\end{eqnarray*}
and  then by combining with (\ref{upper bound}) we get for $\mu$-almost every $x$
\begin{eqnarray*}
h(f,B_\infty(f,x,\vep)) & \leq & \min\left(\max\left(
(e^{M_{\mu,\vep}}-1)\log\vep+\log
^+\|f\|_{l},\,\,0\right),M_{\mu,\vep}\right).
\end{eqnarray*}

Now  we maximize the right hand
side in $M_{\mu,\vep}$. It is maximal when $M_{\mu,\vep}=a$ where $a$ is the solution of
$(e^{a}-1)\log \vep +\log^+ \|f\|_{l}=a$.  We have therefore
\begin{eqnarray*}
a\log\vep +\log^+ \|f\|_{l}&\geq& a;\\[2mm]
a&\leq& \frac{\log^+\|f\|_{l}}{1-\log\vep}\leq
\frac{\log^+\|f\|_{l}}{|\log\vep|}.
\end{eqnarray*}
The proof of Theorem \ref{quasionedim} is completed. \hfill $\Box$

\begin{Rem}The idea of the proof of Theorem
\ref{quasionedim} is related with  the strategy to prove the
existence of symbolic extensions for $C^r$ interval maps in
\cite{DM09}. The production of local entropy by monotonic  branches
is somehow counterbalanced by the decreasing of the Lyapunov
exponents.
\end{Rem}

\noindent{\it Proof of Theorem \ref{wmulti}}.\quad
The proof is very similar to this of Theorem \ref{quasionedim}. As we consider $\vep<L(f)$, any $\vep$ ball meets at most two $f$-monotone branches. Therefore, with the notations of the above proof, we have $M_{x,\vep}=1$ or $2$  for any $x\in [0,1]$ and for any $\vep<L(f)$. Equation (\ref{dm}) may be replaced in this case with $N_x^n:=\sharp\{0\leq k<n, \ M_{f^kx, \vep}=2\}$ by
$$\prod_{k=0}^{n-1}\|f'|_{B(f^kx,\vep)}\|  \leq  \|f'\|_\infty^{n-N_x^n}w(f',\vep)^{N_x^n}.$$
Therefore for $\mu$-almost every $x$ we get
$$h(f,B_\infty(f,x,\vep))\leq \min\left(\left(1-\frac{M_{\mu,\vep}}{\log 2}\right)\log^+\|f'\|+\frac{M_{\mu,\vep}}{\log 2}\log w(f',\vep),M_{\mu,\vep}\right)$$
which leads after optimization to  $$h^*(f,\vep)\leq
\frac{\log2\cdot\log^+ \|f'\|}{\log^+\left( 1/w(f',\vep)\right)}.$$

\begin{Rem}
We only state a rate of convergence for $C^1$ smooth maps in Section \ref{state}. For general continuous multimodal maps, the rate of convergence to zero of the $\vep$-tail entropy may be bounded from above as follows,
$$h^*(f,\vep)\leq \frac{\log 2}{p_\vep}$$
where $p_\vep$ is the largest integer $p$ such that the minimal length of $f^p$-monotone branches, $L(f^p)$, is  larger than $\vep$.

Indeed as in the previous proof of  Theorem \ref{wmulti}, we have
for any multimodal maps $g$, $h^*(g,\vep)\leq \log 2$ for all
$0<\vep<L(g)$. Then by applying this fact to $f^{p_\vep}$ and Lemma
\ref{htailpower} we get
\begin{eqnarray*}
h^*(f,\vep)&\leq &h^*(f^{p_\vep},\vep)/p_{\vep}\\[2mm]
&\leq& \log 2/p_{\vep}.
\end{eqnarray*}

For the tent map, $T(x)=2\max(x,1-x)$, one easily gets that $p_\vep$
is the integer part of $|\log \vep|/\log 2$ and therefore
$h^*(T,\vep) \leq \log4/|\log\vep|$ (one can also prove as in
Theorem \ref{homone} that $h^*(T,\vep)\geq  \log 2/|\log\vep|$).
However it seems quite hard to estimate $p_\vep$ for general
continuous   multimodal maps.
\end{Rem}

\noindent{\it Proof of Theorem \ref{homone}}.\quad To simplify the
exposition we assume $f$ is a $C^2$ unimodal map with a
nondegenerate critical point $c$ (of order $2$) and $\Lambda=P$ is a
hyperbolic repelling fixed point. We call a $2$-horseshoe for $f^p$
a pair of two closed disjoint intervals $J_0,J_1$ such that
$f^p(J_k)\supset J_0\cup J_1$ for $k=0,1$. It is well known that the
$f^p$-invariant set associated to $J_0\cup J_1$ is conjugated  to
the $2$-shift. In particular if $f^l(J_0), f^l(J_1)$ have diameter
less than $\varepsilon$ for all $l=0,...,p$ it will imply that
$h_{\loc}(\mu,\vep)\geq \log 2/p$ with $\mu$ a measure of maximal
entropy of this horseshoe and therefore $h_{\loc}(f,\vep)\geq
\log2/p$. We will prove for any $\vep>0$ the existence of such a
$2$-horseshoe for $f^p$ with $p\leq C|\log\vep|$. The presence of a
horseshoe for interval maps with an homoclinic tangency has
previously been studied by Block in \cite{Block}.

We let $I_\vep$ be the maximal neighborhood of $c$ in
$[c-\vep,c+\vep]$ such that the two connected components of $f$ are
mapped by $f$ on the same interval. Note that $I_\vep$ is of the
form either $[c-\vep',c+\vep]$ or $[c-\vep,c+\vep']$ with $\vep'\leq
\vep$.  For $\vep$ small enough $f^kI_\vep$ has $P$ on its boundary
(recall $f^k(c)=P$ and $c$ is a local extremum of $f$) and its
length is of order $\vep^2$ as $f''(c)\neq 0$. As $c$ belongs to the
unstable manifold of $P$ we may also choose $\vep$ so small  that
$I_\vep\subset W^u(P)$ and then  $l$ large enough such that
$f^{-l}(I_\vep)\subset f^kI_\vep$. For all integers $n$ we have
$f^{-n}(I_\vep)\in B(P,C'e^{-n\lambda(P)/2})$ with
$e^{\lambda(p)}=|f'(P)|>1$
 so  it is enough to take $l=C''|\log\vep|$ for some constant $C''$ independent of $\vep$.
Then one can take $\delta_0,\delta_1>0$ small enough such that the
two connected components $I_\vep\setminus [c-\delta_0,c+\delta_1]$
have the same image by $f$ and
 $f^{-l}(I_\vep)\subset  f^k(I_\vep\setminus [c-\delta_0,c+\delta_1])$.
  This defines a $2$-horseshoe for $f^{k+l}$. For a general hyperbolic repeller one uses Lemma \ref{bound period}.
  It will be explained in details in the next section for surface diffeomorphisms.
  As the argument is the same we do not reproduce it here.  \hfill
  $\Box$ \\

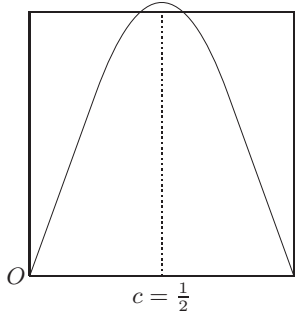
\begin{figure}[h]
\begin{center}
\begin{picture}(200,120)(-70,-12)

\put(-50,0){\line(1,0){100}}

\put(-50,100){\line(1,0){100}}

\put(-50,0){\line(0,1){100}}

\put(50,0){\line(0,1){100}}

\multiput(0,0)(0,2){50}{\line(0,1){0.5}}

\spline(-50,0)(0,138)(50,0)

\put(-55,0){\makebox(0,0){\small $O$}}

\put(0,-8){\makebox(0,0){\small $c=\frac12$}}

$$\, $$

\end{picture}
$$\, $$
\end{center}
\caption{$f(x)=4x(1-x)$ with a homoclinic tangency $c=\frac12$}
\end{figure}

\noindent{\it Proof of Proposition \ref{rrrem}}.\quad
 We consider the quadratic map $f_4$, $f_4(x)=4x(1-x)$. We assign to any  $f_4^n$-monotone branch
$I_n$ an element  $a(I_n)$ of $\{0,1\}^n$, as follows $(a(I_n))_k=0$
if $f^k(I_n)\subset [0,1/2]$ and $(a(I_n))_k=1$ if not. We also let
$x(I_n)$ be the center of $I_n$.  We consider the subshift $Y_p$ of
finite type of $\{0,1\}^\NN$ where we have forbidden the word
$\underbrace{010...0}_{p}$  which correspond to the $f^p$-monotone
branch with the critical point $1/2$ on its right boundary. This $f^p$-monotone branch
has length $\vep:=\vep_p$ with $|\log \vep_p|\underset{p}{\sim}p\log 2$ : indeed the length of
the $f^{p-1}$-monotone branch associated to
$\underbrace{10...0}_{p-1}$ has length $\vep'_p$ with $|\log \vep'_p|\underset{p}{\sim}p\log 4$ and the tangency
at the critical point is quadratic. We also let $Y_p(n)$ be the set of words of
length $n$ in $Y_p$. Clearly $\{x(I_n), \ a(I_n)\in Y_p\}$ is
$(n,\vep)$  separated. Therefore

\begin{eqnarray*}h(f_4,\vep)&\geq &\limsup \frac{1}{n}\log \sharp Y_p(n);\\
& \geq & h(\sigma,Y_p).
\end{eqnarray*}
Finally we have

\begin{eqnarray*}
h(\sigma,Y_p)&=&h(\sigma^p,Y_p)/p;\\
&=&\log (2^p-1)/p;\\
&=&\log 2 -\frac{1}{p2^p}+o(\frac{1}{p2^p}).
\end{eqnarray*}
We conclude that $h(f_4)- h(f_4,\vep)=\log 2- h(f_4,\vep)=
o\left(\frac{\vep^\alpha}{|\log \vep|}\right)$ for any $0<\alpha<
1$.

\section{Modulus of continuity of the topological entropy : proof of Proposition \ref{mod} and some examples}\label{modu}

\noindent{\em Proof of Proposition \ref{mod}:}
 For any $\vep>0$ and any $g\in G$ we have
\begin{eqnarray*}
h(g)&\leq &h(g,\vep)+h_{\loc}(g,\vep);\\[2mm]
&\leq &h(g,\vep)+h^G_{\loc}(\vep);\\[2mm]
&\leq & \frac{1}{p_\vep}\log r_{p_\vep}(g,\vep/2)+h^G_{\loc}(\vep).
\end{eqnarray*}
Now one easily checks by induction on $k$ that
  $d(f^k,g^k)\leq d(f,g)\sum_{l=0}^{k-1} M_0^l\leq  \vep/4$ for any $k=1,\cdots,p_\vep$
 once we have $d(f,g)\leq \frac{\vep}{4} M_0^{-p_\vep}$. Indeed for all $x\in M$ we have
\begin{eqnarray*}
d(f^kx,g^kx)&\leq & d(gf^{k-1}x,g^kx)+d(f^{k}x,gf^{k-1}x).
\end{eqnarray*}
and then by induction hypothesis \begin{eqnarray*}
d(f^kx,g^kx)&\leq & M_0d(f^{k-1}x,g^{k-1}x)+d(f,g);\\[2mm]
&\leq & d(f,g)\sum_{l=0}^{k-1} M_0^l;\\[2mm]
&\leq & \frac{\vep}{4} M_0^{-p_\vep} \frac{M_0^{p_\vep}-1}{M_0-1}\leq \frac{\vep}{4}.
\end{eqnarray*}
In this case we have then $r_{p_\vep}(g,\vep/2)\leq
r_{p_\vep}(f,\vep/4)$ and finally we obtain according to the choice
of $p_\vep$ :

\begin{eqnarray*}
h(g)&\leq &
\frac{1}{p_\vep}\log r_{p_\vep}(f,\vep/4)+h^G_{\loc}(\vep);\\[2mm]
&\leq &  h(f,\vep/4)+2h^G_{\loc}(\vep);\\[2mm]
&\leq & h(f)+2h^G_{\loc}(\vep).
\end{eqnarray*}
This concludes the proof of Proposition \ref{mod}.\hfill $\Box$

A continuous dynamical system $f$ is said to satisfy the property
$(P)$ if for $\vep$ small enough we have
 $$\frac{1}{n}\log r_n(f,\vep)-h(f,\vep) \simeq \frac{|\log \vep |}{n},$$
i.e. there exists $C>1$ and $\zeta(f)>0$ such that for all
$\zeta(f)>\vep>0$ and for all integers $n$ we have
$$ \frac{|\log \vep |}{Cn}\leq \frac{1}{n}\log r_n(f,\vep)-h(f,\vep) \leq \frac{C|\log \vep |}{n}.$$
One easily sees this is the case of the following zero topological
entropy systems : the identity map,
 translation maps, interval and circles homeomorphisms,... Yomdin also proved in \cite{Yom91} that  a  polynomial of degree $k$ on a compact invariant set of
$\mathbb{R}^2$ of maximal entropy $\log k$ also satisfies this
property.

\begin{Que}
What are the dynamical systems satisfying property (P)? Does it
contain a large class of systems?
\end{Que}

We will study the modulus of continuity of the topological entropy
for systems in $C^\mathcal{M}$ $C^0$-close to a system satisfying
the property (P). To simplify we will only consider surface
$V$-ultradifferentiable maps and  the limit case in Theorem
\ref{surfh}, $M_k=M_0k^{k^2}$ for all integers $k$ where $M_0$ is
some fixed real number larger than $e$.

\begin{Cor}\label{upper contin}
Let $(f,M)$ be a continuous  dynamical system satisfying property $(P)$ with $M$ a smooth compact Riemannian surface.
 Then there exists a  constant $C=C(f)$, such that  for all  $0<\vep<\min(1,R_{inj}^2)$ and for all
  $g\in C_V^{(M_0k^{k^2})_k}(M)$ with $d_{C^0}(f,g)\leq \vep$ :
$$h(g)\leq h(f)+C\log M_0\sqrt{\frac{\log|\log\vep|}{|\log\vep|}}.$$
\end{Cor}

\begin{proof}
 With the notation of Proposition \ref{mod} we have
 $p_\vep\simeq |\log (\vep/4)|/h_{loc}^G(\vep)$. We assume now $G=C_V^\mathcal{M}(M)$
 with $M_k=M_0k^{k^2}$ for all $k\in \mathbb{N}$. Then by Theorem \ref{gen},
  we can take $h^G_{loc}(\vep)= C_1\log M_0\frac{\log|\log \vep|}{|\log
  \vep|}$ for some universal constant $C_1$.
  Thus, with   $\delta_\vep:=\frac{\vep}{4} M_0^{-p_\vep}$, we have clearly $|\log \delta_\vep|\simeq \frac{|\log \vep|^2}{\log|\log\vep|}$  and  $\log |\log \delta_\vep| \simeq \log |\log \vep|$.
  It follows that

\begin{eqnarray*}
h^G_{\loc}(\vep)&\leq &  C\log M_0\frac{\log|\log \vep|}{|\log
\vep|};\\[2mm]
& \leq & C\log M_0 \frac{\log|\log \delta\vep|}{\sqrt{|\log
\delta_\vep|\times\log|\log\delta_\vep|}};\\[2mm]
&\leq & C\log M_0\sqrt{\frac{\log|\log \delta_\vep|}{|\log
\delta_\vep|}},
\end{eqnarray*}
for some  $C=C(f)$.  Therefore for $g\in C^\mathcal{M}(M)$ with
$d(f,g)\leq \delta$ we get by applying Proposition \ref{mod}
$$h(g)\leq h(f)+ 2C\log M_0\sqrt{\frac{\log|\log \delta|}{|\log \delta|}}.$$
\end{proof}

\section{$C^r$ $(r\geq 2)$  robust  examples}\label{robust}

In this section, we construct non $h$-expansive  $C^r$ $(r\geq 2)$
open domains associated with homoclinic tangencies to prove Theorem
\ref{non EE}.

\subsection{Structure of hyperbolic sets}  We first make some definitions.  Fix
$f\in \Diff^r(M)$ with $r \geq 1$.  Let $\Lambda \subset M$ be an
$f$-invariant set.  We call $\Lambda$ a hyperbolic
 set for $f$ if  there exist  $\lambda_0\in (0,1)$, $C>0$,
  and a $Df$-invariant decomposition $T_{\Lambda}M=E^s\oplus E^u$ such that
\begin{eqnarray*}
\|D_xf^nv\|\leq C\lambda_0^n\|v\|,\quad\text{for any}\,\,n\geq
0,\,\,v\in
E^s(x),\,\,x\in \Lambda;\\[2mm]
\|D_xf^{-n}v\|\leq C\lambda_0^n\|v\|,\quad\text{for any}\,\,n\geq
0,\,\,v\in E^u(x),\,\,x\in \Lambda.
\end{eqnarray*}
At most taking a suitable equivalent metric, we can assume $C=1$ in
above definition.  $\Lambda$ is further called a basic set if
\begin{itemize}\item $\Lambda$ is transitive:  there exists $x\in \Lambda$ whose orbit is dense in $\Lambda$;\\
\item   $\Lambda$ is isolated:  there exists a
neighborhood $U$ of $\Lambda$ such that $$\bigcap_{\,n\in
\mathbb{Z}}f^n(U) = \Lambda.$$   Here $U$ is called an adapted
neighborhood of $\Lambda$.  \end{itemize}For a hyperbolic set
$\Lambda$, given a point $x\in \Lambda$, there exist $C^r$
injectively immersed sub-manifolds $W^s(x)$ and $W^u(x)$ given by
$$W^s(x) = \{y \in M : d( f^n(y), f^n(x))\to 0 \,\,\text{as}\,
n\to+\infty\}$$ and
$$W^u(x) = \{y \in M : d(f^{-n}(y), f^{-n}(x))\to
0\,\,\text{as}\,n\to +\infty\},$$ see for example,  Theorem 3.2 in
\cite{HP}.  Here $W^s(x)$, $W^u(x)$ are called the stable manifold
and the unstable manifold at $x$, respectively. Furthermore, the
stable manifold of size $\delta>0$ is defined by
$$W^s_{\delta}(x) = \{y \in M : d( f^n(y), f^n(x))\leq \delta \,\,\text{for all}\,
n\geq 0\}.$$ Similarly,  one can define the unstable manifold of
size $\delta$ as $W^u_{\delta}(x)$ by considering $f^{-1}$.

For $x,y\in \Lambda$, and a point $z\in W^u(x)\cap W^s(y)$, we call
$z$ is a transversal intersection point if
$$T_zW^u(x)\oplus T_z W^s(x)=T_z M.$$
Conversely, a non transversal intersection point is called a
tangency.

A periodic point $p$ of $f$ is a point such that there is a positive
integer $n$ with $f^n(p)=p$, where $n$ is called a period of $p$.
The periodic point $p$ is hyperbolic if all eigenvalues of the
derivative $Df^{n}(p)$ have modulus different from 1. In fact, a
periodic point $p$ is hyperbolic if and only if the orbit of $p$ is
a hyperbolic basic set.   More generally, we say that an
$f$-invariant set $\Lambda$ is periodic if there exist a subset
$\Lambda_1\subset \Lambda$ and a positive integer $n$ such that
\begin{itemize}
\item \quad  $f^{n}(\Lambda_1) =\Lambda_1$,\\
\item \quad  $\Lambda =\bigcup_{0\leq i<n} f^i(\Lambda_1)$. \end{itemize}

\noindent{In this case}, we call $n$ to be a period of $\Lambda$,
and $\Lambda_1$ to be a base of $\Lambda$. Denote the diameter of
$\Lambda$ in the base $\Lambda_1$ by
$$\diam_{\Lambda_1}(\Lambda)=\max_{0\leq i<n}\,\diam (f^i(\Lambda_1)).$$

By the uniform hyperbolicity of $\Lambda$, there exist
$\varepsilon_0,\delta_0>0$, $\lambda\in (0,1)$ such that
\begin{itemize}\item
\begin{eqnarray*}d(f^n(y),f^n(z))&\leq& \lambda^nd(y,z), \quad \text{for all}\, n\geq 0,\,\,y,z\in W^s_{\vep_0}(x),\,x\in
\Lambda;\\[2mm]
d(f^{-n}(y),f^{-n}(z))&\leq& \lambda^nd(y,z), \quad \text{for all}\,
n\geq 0,\,\,y,z\in W^u_{\vep_0}(x),\,,x\in \Lambda;\end{eqnarray*}
\item $W^s_{\varepsilon_0}(x)\cap W^u_{\varepsilon_0}(y)$ contains a
single point $[x,y]$ whenever $d(x,y)<\delta_0$.  Furthermore, the
function
$$[\cdot, \cdot]: \{(x,y)\in M\times M\mid d(x,y)<\delta_0 \}\rightarrow M$$
is continuous. \end{itemize}

 A  rectangle $R$ is understood by a subset of $M$ with diameter
 smaller than $\vep_0$
 such that  $[x,y]\in R$ whenever $x,y\in R$.
 For $x\in
R$ let $$W^s(x,R)=W^s_{\varepsilon_0}(x)\cap R\quad \mbox{and}\quad
W^u(x,R)=W^u_{\varepsilon_0}(x)\cap R.$$ For a hyperbolic basic set
$\Lambda$,  one can obtain the following structure known as a Markov
partition $\mathcal{R}=\{R_1,R_2,\cdots,R_l\}$ of $\Lambda$ with
properties:
\begin{enumerate}
\item[(i)] $\Int R_i\cap \Int R_j=\emptyset$ for $i\neq j$;\\
\item[(ii)] $f W^u(x,R_i)\supset W^u(fx,R_j)$ and \\
$f W^s(x,R_i)\subset W^s(fx,R_j)$ when $x\in \Int R_i$, $fx\in \Int
R_j$,
\end{enumerate}
See  Bowen\cite{Bowen}. Using the Markov Partition $\mathcal{R}$ one
can define the transition matrix $A=A(\mathcal{R})$ by

$$A_{i,j}=\begin{cases}1\quad \mbox{if}\quad  \Int R_i\cap f^{-1} ( \Int R_j)\neq \emptyset;\\ 0\quad \mbox{otherwise}.
\end{cases}$$
The subshift $(\Sigma_A,\sigma)$ associated with $A$ is given by
$$\Sigma_A=\{\underline{q}\in \Sigma_l\mid \,A_{q_i,q_{i+1}}=1\quad \forall i\in \mathbb{Z}\}.$$
For each $\underline{q}\in \Sigma_A$,  the set $\cap_{i\in
\mathbb{Z}}f^{-i}R_{q_i}$ contains of a single point, which we
denote by $\pi_0( \underline{q} )$.  We define
$$\Sigma_A(i)=\{\underline{q}\in \Sigma_A\mid q_0=i\}.$$
The following properties hold  for the map $\pi_0$ (see Theorem 28
of \cite{Bowen}):
\begin{enumerate}
\item[(i)] The map $\pi_0: \Sigma_A\rightarrow \Lambda$ is a continuous surjection
satisfying $\pi_0\circ \sigma=f\circ \pi_0;$\\
\item[(ii)] $\pi_0(\Sigma_A(i))=R_i\cap \Lambda$,\quad $1\leq i\leq l$.
\end{enumerate}
Since $\Lambda$ is a hyperbolic basic set, by Smale's Spectral
Decomposition Theorem \cite{Smale}, there exists $n_0\in \mathbb{N}$
such that
\begin{eqnarray*}&&\Lambda=\Lambda_1\cup\cdots\cup\Lambda_{n_0},\quad
\Lambda_i\cap\Lambda_j=\emptyset,\,\,1\leq i< j\leq n_0,\\[2mm]
&&f^i(\Lambda_1)=\Lambda_{1+i}, \,\,1\leq i\leq n_0-1,\quad
f^{n_0}(\Lambda_1)=\Lambda_1.\end{eqnarray*}Moreover, $f^{n_0}$ is
mixing in $\Lambda_1$, i.e., given pairs of open sets $U_1, U_2$
with nonempty intersections with $\Lambda_1$, $\exists\, n_1\in
\mathbb{N}$, s.t. $f^{n_0n_1}(U_1)\cap U_2\neq \emptyset$,
$\forall\,n\geq n_1$. Equivalently to say here, for the transition
matrix $B$ of a Markov partition $\mathcal{R}$  for
$f^{n_0}\mid_{\Lambda_1}$, one can find $n_1\in \mathbb{N}$ such
that all elements of the matrix $B^{n_1}$ are positive.

\begin{Lem}\label{bound period}  There exists $\vep_1>0$ such that for any  $\vep\in(0,\vep_1)$ and $x_1,x_2\in
\Lambda$, one can find a  periodic point $p\in \Lambda$ with a
period $\tau(p)\in [2|\log\vep|/|\log \lambda|,9|\log\vep|/|\log
\lambda|]$ such that
$$d(p,x_1)\leq \vep,\quad d(f^i(p),x_2)\leq \vep\quad \text{for some }\,\,i\in [0,\tau(p)].$$
\end{Lem}
\begin{proof}For $x_1,x_2\in \Lambda$, we can choose $m_1,m_2\in [0,n_0-1]$ such
that $$y_1:=f^{-m_1}(x_1)\in \Lambda_1, \quad y_2:=f^{-m_2}(x_2)\in
\Lambda_1.$$ Let $g=f^{n_0}$. Take $\underline{q}, \underline{q}'\in
\Sigma_B$ with $y_1=\cap_{i\in \mathbb{Z}}g^i(R_{q_i})$,
$y_2=\cap_{i\in \mathbb{Z}}g^i(R_{q'_i})$. Since
 all elements of the matrix
$B^{n_1}$ are positive,  for any $n\geq n_1$ there exists a sequence
$i_1,i_2,\cdots, i_{n_1-1}$, $i_1',i_2',\cdots, i_{n_1-1}'$ such
that
$$B_{q_n,i_1}B_{i_1,i_2}\cdots
B_{i_{n_1-2},i_{n_1-1}}B_{i_{n_1-1},q'_{-n}}>0,\,
B_{q'_n,i_1'}B_{i'_1,i'_2}\cdots
B_{i'_{n_1-2},i'_{n_1-1}}B_{i'_{n_1-1}q_{-n}}>0$$ which imply the
following periodic point is contained in $\Sigma_B$:
$$w:=[q_{-n},\cdots, q_{-1}, \stackrel{0}{q_0};
q_1,\cdots, q_n, i_1,\cdots, i_{n_1-1}, q'_{-n},\cdots,
q'_0,q_1',\cdots,q'_n,i_1',\cdots,i_{n_1-1}'].$$ Let
$p=\cap_{i\in\mathbb{Z}}g^{i}(R_{w_i})$, which is a periodic point
of $g$ with a period $4n+2n_1$.  Then for each $i\in[-n, n] $,
$g^{i}(p),g^{i}(y_1)$ belong to the same rectangle in the Markov
partition $\mathcal{R}$. Also, $g^{i}(g^{2n+n_1}(p)),g^{i}(y_2)$
belong to the same rectangle of  $\mathcal{R}$. They imply
$$d(\pi^{s/u}_{g^{i}(y_1)}(g^{i}(p)),g^{i}(y_1))\leq
\vep_0,\quad
d(\pi^{s/u}_{g^{i}(y_2)}(g^{i}(g^{2n+n_1}(p))),g^{i}(y_2))\leq
\vep_0,
$$ for $ i\in [-n, n]$,  where $\pi^{s/u}_x(z)$ denotes the
intersection point of $W^{u/s}_{\vep_0}(z)$ and
$W^{s/u}_{\vep_0}(x)$. By the uniform hyperbolicity of $\Lambda$,
$$d(\pi^{s/u}_{y_1}(p),y_1)\leq \vep_0\lambda^{nn_0},\quad d(\pi^{s/u}_{y_2}(g^{2n+n_1}(p)),y_2)\leq \vep_0\lambda^{nn_0}.$$
Note that there exists $C_0>0$ such that $d(x,z)\leq
C_0\max(d(\pi^{s}_{x}(z),x), d(\pi^{u}_{x}(z),x))$ for any $z$ with
$d(z,x)\leq \delta_0$, $x\in \Lambda$.  We deduce
$$d(p,y_1)\leq C_0\vep_0\lambda^{nn_0},\quad d(g^{2n+n_1}(p),y_2)\leq C_0\vep_0\lambda^{nn_0}.$$
 Choose $\vep_1>0$ such
that $\max\{|\log(C_0\vep_0)|+n_0\log
\|Df\|,\,\,n_0|\log\lambda|\}\leq \frac14|\log\vep_1|$. For any
$\vep\in (0,\vep_1)$, let
$$n=\Big{\lfloor}\frac{|\log\vep-\log(C_0\vep_0)-n_0\log
\|Df\||}{n_0|\log\lambda|}\Big{\rfloor}+1\in
\Big{[}\frac{|\log\vep|}{2n_0|\log\lambda|},\frac{3|\log\vep|}{2n_0|\log\lambda|}\Big{]}.$$
Observe that $w$ has a period  $4n+2n_1\in [4n, 6n]$. Then $p$ as a
periodic point of $f$ has a period
$$\tau(p)\in [4n_0n, 6n_0n]\subset \Big{[}\frac{2|\log\vep|}{|\log
\lambda|},\frac{9|\log\vep|}{|\log\lambda|}\Big{]}$$  and
$$\max\Big{\{}d(p,y_1),\,\, d(f^{n_0(2n+n_1)}(p),y_2)\Big{\}}\leq C_0\vep_0\lambda^{nn_0}\leq \vep\|Df\|^{-n_0}.$$
Hence, \begin{eqnarray*}d(f^{m_1}(p), x_1)&=& d(f^{m_1}(p),
f^{m_1}(y_1)) \leq \|Df\|^{m_1}d(p, y_1)  \leq
\vep\\[2mm]d(f^{m_2+n_0(2n+n_1)}(p),x_2)&=&
d(f^{m_2+n_0(2n+n_1)}(p),f^{m_2}(y_2))\\[2mm]
 &\leq& \|Df\|^{m_2}d(f^{n_0(2n+n_1)}(p),y_2)\leq \vep.\end{eqnarray*} Moreover,
$0\leq m_2+n_0(2n+n_1)-m_1\leq 4n_0n\leq \tau(p)$. The proof of
Lemma \ref{bound period}   is completed.
\end{proof}
 The following  Proposition states
that the uniformly hyperbolic structure holds in a persistent way.
\begin{Prop}\label{prop2} Let $\Lambda=\Lambda(f)$ be a hyperbolic basic set
for the $C^1$ diffeomorphism $f$ on $M$ with adapted neighborhood
$U$. Given $C> 0$, there is a neighborhood $\mathcal{N}_C$ of $f$ in
$\Diff^1(M)$ such that if $g\in \mathcal{N}_C$, then $\Lambda(g)
=\cap_{n\in \mathbb{Z}}\, g^n(U)$ is a hyperbolic basic set for $g$
and there is a unique continuous embedding $h_g :\Lambda(f) \to M$
such that $h_g(\Lambda(f))= \Lambda(g)$, $g\circ h_g = h_g\circ f$
and $d(h_g,\id)<C$. Moreover, $h_f = \id$.
\end{Prop}

\subsection{Thickness of Cantor sets} Let $K$ be a cantor set, i.e., a compact perfect totally disconnected subset  of
$\mathbb{R}$.  Let $K_0$ be the smallest closed interval containing
$K$. Then $K_0-K=\cup_{i=0}^{\infty}U_i$, where $U_i\cap
U_j=\emptyset$ if $i\neq j$ and each $U_i$ is a bounded open
interval. Let $U_{-2}$, $U_{-1}$ be the unbounded components of
$\mathbb{R}\setminus K$. All $U_i$, $i\geq -2$,  are called the gaps
of $K$. For any $i\geq 1$, define  $K_i=K_0\setminus\left(\cup_{0\leq j\leq
i-1}U_j\right)$. Then
$$K_0\supseteq K_1\supseteq \cdots \supseteq \cdots.$$
Each $K_i$ is a union of closed intervals and $K=\cap_{i\geq 0}
K_i$. We call $\{K_i\}_{i\geq 0}$ to be a defining sequence for $K$.
For $i\geq 1$, let $K_i^*$ be the  connected component of $K_i$ containing
$U_i$, then $K^*_i\setminus U_i$ is the union of two closed intervals $I_i^{l}$,
$I_i^{r}$.
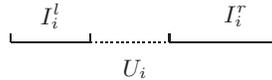
\begin{figure}[h]
\begin{center}
\begin{picture}(200,50)(-100,-12)

\put(-50,0){\line(1,0){30}}

\multiput(-20,0)(2,0){15}{\line(1,0){0.5}}

\put(10,0){\line(1,0){40}}

\put(-20,0){\line(0,1){3}}

\put(-50,0){\line(0,1){3}}

\put(10,0){\line(0,1){3}}

\put(50,0){\line(0,1){3}}

\put(-3,-10){\makebox(0,0){\small $U_i$}}

\put(-35,10){\makebox(0,0){\small $I_i^l$}}

\put(35,10){\makebox(0,0){\small $I_i^r$}}

$$\, $$

\end{picture}
$$\, $$
\end{center}
\caption{Remove open intervals }
\end{figure}

For an interval $I$, denote by $|I|$ the length of $I$. Set
$$\tau(\{K_i\})=\inf_{i\geq 0} \Big{\{} \min
(\frac{|I_i^l|}{|U_i|},\frac{|I_i^r|}{|U_i|})\Big{\}}.$$ The
thickness of $K$ is defined by
$$\tau(K)=\sup\{\tau(\{K_i\}):\, \{K_i\} \,\, \text{is a defining sequence for}\,\,  K\}.$$

\begin{Lem}[Gap lemma, Lemma 4 of \cite{New79}]\label{gap lemma} Let $K, F$ be two cantor sets  with
thicknesses $\tau_1, \tau_2$. If $\tau_1\cdot \tau_2>1$, then one of
the following alternatives occurs:
\begin{itemize}\item $K$ is contained  in a gap closure of $F$;\\
\item $F$ is contained  in a gap closure of $K$;\\
\item $K\cap F\neq \emptyset$. In this case, for any defining sequences
$\{K_i\}$ of $K$, $\{F_i\}$ of $F$ with
$\tau(\{K_i\})\cdot\tau(\{F_i\})>1$, it holds that $\Int(K_i\cap
F_i)\neq \emptyset$ for any $i$.
\end{itemize}

\end{Lem}

Let $\Lambda$ be a hyperbolic basic set of $f\in \Diff^2(M)$ and $p$
be a periodic point of $f$. We can parameterize  $W^s(p)$ and
$W^u(p)$ such that $f\mid_{ W^s(p)}$ and $f\mid_{ W^u(p)}$
 are linear, see \cite{Sternberg}. We define the unstable thickness of
 $(\Lambda,p)$ as $\tau^u(\Lambda,p)=\tau(W^s(p)\cap \Lambda)$, the  stable thickness of
 $(\Lambda,p)$ as $\tau^s(\Lambda,p)=\tau(W^u(p)\cap
 \Lambda)$.  Observe that $W^s(p)\cap \Lambda$ is $f$-invariant, and
 $f\mid_{ W^s(p)}$ is linear, there exist arbitrarily   small compact
 neighborhoods $K$ of $p$ in $W^s(p)\cap \Lambda$ such that $\tau(K)=\tau(W^s(p)\cap
 \Lambda)=\tau^u(\Lambda,p)$.  The same argument applies to
 $\tau^s(\Lambda,p)$. It can be shown that $\tau^{s/u}(\Lambda,p)$ is
 independent of $p$ (Proposition 5 in \cite{New79}). We denote
 $\tau^{s/u}(\Lambda)=\tau^{s/u}(\Lambda,p)$.
By Proposition \ref{prop2}, the persistence  of $\Lambda$ holds in a
$C^1$ neighborhood $\mathcal{N}_1$ of $f$. Furthermore,

\begin{Prop}[Proposition 6 in \cite{New79} or Theorem 2 of Chapter 4.3 in \cite{PT93}]\label{prop3}
There exists a $C^2$ neighborhood $\mathcal{N}_2\subset
\mathcal{N}_1$ of $f$ such that the thicknesses
$\tau^{s/u}(\Lambda(g))$ depend continuously for $g\in
\mathcal{N}_2$.
\end{Prop}

\subsection{Small Horseshoes}

Let $\Lambda_0$ be a hyperbolic basic set of $f\in
\Diff^r(M)$ whose stable manifolds and unstable manifolds tangent at
some point. Then by  Lemma 7 and Lemma 8 of \cite{New79} we can at
most by a $C^r$ perturbation  let $f$ have  a hyperbolic basic set
$\Lambda$ satisfying $\tau^{s}(\Lambda)\cdot\tau^u(\Lambda)>1$ and
containing a periodic point $p\in \Lambda$ with a tangency $x_0$ of
$W^u_f(p)$ and $W^s_f(p)$.   By Proposition \ref{prop3}, there
exists a $C^r$ neighborhood $\mathcal{N}_2 $ of $f$ such that
$\tau^{s}(\Lambda(g))\cdot\tau^u(\Lambda(g))>1$  for $g\in
\mathcal{N}_2$.

 For each $g\in \mathcal{N}_2$, take a $C^1$ stable
foliation $\mathcal{F}^s_g(U_1)$  in a neighborhood $U_1$
 of $\Lambda(g)$ such that for $x\in \Lambda(g)$, the leave
$\mathcal{F}^s_g(x)$ is a subset of $W^s_g(x)$.
$\mathcal{F}^s_g(U_1)$ varies continuously with respect to $g\in
\mathcal{N}_2$. Similarly, we have a $C^1$  unstable foliation
$\mathcal{F}^u_g(U_1)$. See the constructions of stable and unstable
foliations in Section 3, Chapter 2 of \cite{PT93}.

For the tangency point $x_0\in W^u_f(p)\cap W^s_f(p)$,
$T_{x_0}W^u_f(p)=T_{x_0}W^s_f(p).$  We let the tangency at $x_0$ is
quadratic (like $y=ax^2$ near the tangency point). Otherwise we can
obtain this with an arbitrarily $C^r$ small perturbation. Denote
$$L=\max\big{\{}d_s(p,x_0), d_u(p,x_0)\big{\}}$$ where $d_{s/u}$ are the distances in the leaves of $\mathcal{F}^{s/u}$.
 So, for $g$ $C^r$ close to $f$,  we can take a $C^1$ line
$l(g)$ near $x_0$ consisting of tangencies of $\mathcal{F}^s_{g}(U)$
and $\mathcal{F}^u_{g}(U)$ with the transversal property:
$$T_{x}l(g)\oplus T_{x}\mathcal{F}^u_g(x)=T_{x}M,\quad \forall\,x\in l(g).$$
Now for small $\delta_2>\delta_1>0$, and $g$ $C^r$ close to $f$,
define projections
\begin{eqnarray*}\pi_1(g): W^s_{\delta_1}(p(g))&\rightarrow& l(g),\\[2mm]
\pi_2(g): W^u_{\delta_1}(p(g))&\rightarrow& l(g)
\end{eqnarray*}
which project along leaves of $\mathcal{F}^u_g(y,L+\delta_2)$ and
$\mathcal{F}^s_g(y,L+\delta_2)$, $y\in W^{s/u}_{\delta_1}(p(g))$,
where $\mathcal{F}^{u/s}_g(y,a)$ denote the $a$-disc centered at $y$
in the leaves $\mathcal{F}^{u/s}_g(y)$. Here $\pi_1(g), \pi_2(g)$ are $C^1$ and continuous in $g$.

\begin{figure}[h]
\begin{center}
\begin{picture}(200,120)(-30,-12)

\put(-20,0){\line(1,0){50}}

\put(-20,0){\line(0,1){50}}

\put(-20,50){\line(1,0){50}}

\put(30,50){\line(0,-1){50}}

\put(-45,25){\line(1,0){200}}

\put(-45,22){\line(1,0){200}}

\put(-45,19){\line(1,0){200}}

\put(-45,16){\line(1,0){200}}

\put(-45,28){\line(1,0){200}}

\put(-45,31){\line(1,0){200}}

\put(-45,34){\line(1,0){200}}

{\color{red}

\put(5,-25){\line(0,1){100}}

\put(8,-25){\line(0,1){100}}

\put(11,-25){\line(0,1){100}}

\put(14,-25){\line(0,1){100}}

\put(2,-25){\line(0,1){100}}

\put(-1,-25){\line(0,1){100}}

\put(-4,-25){\line(0,1){100}}

\qbezier(5,75)(5,95)(35,100)

\qbezier(35,100)(90,105)(110,50)

\qbezier(110,50)(130,0)(150,50)

\qbezier(2,75)(2,95)(32,103)

\qbezier(32,103)(87,108)(112,53)

\qbezier(112,53)(130,3)(147,53)

\qbezier(8,75)(8,95)(32,97)

\qbezier(32,97)(93,102)(108,47)

\qbezier(108,47)(130,-3)(153,47)

}

{\color{blue}\put(127,0){\line(0,1){50}}}

\put(50,10){\makebox(0,0){\tiny $\mathcal{F}^s_g(U)$}}

\put(-25,65){\makebox(0,0){\tiny $\mathcal{F}^u_g(U)$}}

 \put(130,60){\makebox(0,0){\tiny $l(g)$}}

$$\, $$

\end{picture}
$$\, $$
\end{center}
\caption{Interval of tangencies}
\end{figure}
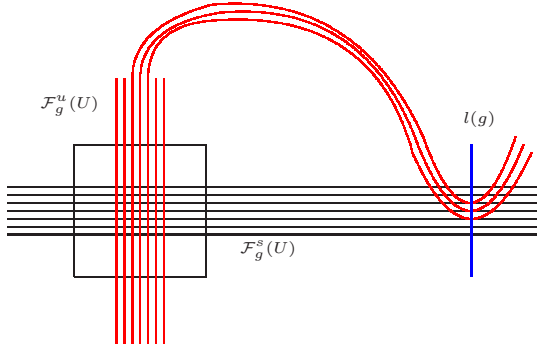

Observing that $\Lambda$ is uniformly hyperbolic, there exists
$L_0>0$ and $\lambda\in (0,1)$, such that
\begin{eqnarray*}d(g^n(x_1),g^n(x_2))\leq\lambda^nd(x_1,x_2), \quad \text{for all}\,\,n\geq
0,\,\,\forall\,x_1,x_2\in W^s_{L_0}(x),\,\,x\in \Lambda(g),\\[2mm]
d(g^{-n}(x_1),g^{-n}(x_2))\leq\lambda^nd(x_1,x_2), \quad \text{for
all}\,\,n\geq 0,\,\,\forall\,x_1,x_2\in W^u_{L_0}(x),\,\,x\in
\Lambda(g).
\end{eqnarray*}

 Since $L$ is fixed, we can take $N\in \mathbb{N}$ and $a_0>0$ such that for any
 $\delta\in (0,L+\delta_1)$,
\begin{itemize}\item $\diam^s(g^n(B^s(x,\delta)))<a_0\delta g^{N}(B^s(x,\delta))
\subset W^s_{L_0}(g^N(y))$, for all $0\leq n\leq N$, for all $x$ in $W^s_{L+\delta_1}(y)$ and for all $y\in \Lambda(g)$,
\item $\diam^u(g^{-n}(B^u(x,\delta)))<a_0\delta \text{ and } g^{-N}(B^u(x,\delta)) \subset W^u_{L_0}(g^{-N}(y))$,
 for all $0\leq n\leq N$, for all $x$  in $W^u_{L+\delta_1}(y)$ and for all $y\in \Lambda(g)$,
\end{itemize}
where $B^{s/u}(z,\delta)$ are  the balls in $W^{s/u}(y)$ centered at
$z$ with radius $\delta$; $\diam^{s/u}$ are the diameters along
$s/u$-leaves. Consequently,  we have
\begin{eqnarray*} \diam^s(g^n(B^s(x,\delta)))<a_0\delta, \quad \text{for all}\,\, n\geq0,\,\,\forall\,x\in W^s_{L+\delta_1}(y),\,\,y\in \Lambda(g),\\[2mm]
\diam^u(g^{-n}(B^u(x,\delta)))<a_0\delta , \quad \text{for
all}\,\,n\geq 0,\,\,\forall\,x\in W^u_{L+\delta_1}(y),\,\,y\in
\Lambda(g).
\end{eqnarray*}

 We give $l(g)$ an orientation so that we can say
up-side and below-side in $l(g)$.  Without loss of generality, we
suppose the leaves of $\mathcal{F}^s_g$ near $l(g)$ are horizontal.
Noting that the tangency $x_0$ is quadratic, we can see all leaves
of $\mathcal{F}^u$ bent upwardly nearby $l(g)$. Thus, there is
$a_1>0$ such that for any $z_1\in l(g)$ and $z_2\in l(g)$ below
$z_1$, the nearby two intersections of $\mathcal{F}^s(z_1)$ and
$\mathcal{F}^u(z_2)$ are contained in a ball with radius
$a_1\sqrt{d(z_1,z_2)}$.

\begin{figure}[h]
\begin{center}
\begin{picture}(200,120)(-30,-12)

\put(-50,0){\line(1,0){100}}

\spline(-50,70)(0,-26)(50,70)

\spline(-50,65)(0,-30)(50,65)

\put(0,-30){\line(0,1){80}}

\put(0,0){\circle*{2}}

\put(0,-4){\circle*{2}}

\put(0,0){\circle{40}}

\put(35,70){\makebox(0,0){\tiny $\mathcal{F}^u(z_1)$}}

\put(65,60){\makebox(0,0){\tiny $\mathcal{F}^u(z_2)$}}

\put(70,0){\makebox(0,0){\tiny $\mathcal{F}^s(z_1)$}}

\put(20,-25){\makebox(0,0){\small $B_{\delta}(z_1)$}}

\put(5,5){\makebox(0,0){\tiny $z_1$}}

\put(5,-8){\makebox(0,0){\tiny $z_2$}}

\put(5,60){\makebox(0,0){\small $l(g)$}}

$$\, $$

\end{picture}
$$\, $$
\end{center}
\caption{The size of the transversal intersection}
\end{figure}
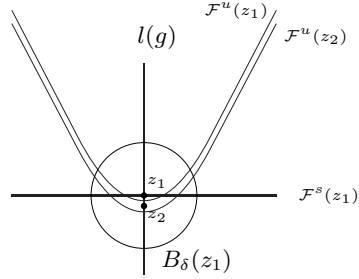

Let $K_1(g)$, $K_2(g)$ be small compact one side neighborhoods of
$p(g)$ in $W^s(p(g))\cap \Lambda(g)$ and $W^u(p(g))\cap \Lambda(g)$,
depending continuously on $g$, and such that
$$\tau(K_1(g))=\tau^u(\Lambda(g)),\quad
\tau(K_2(g))=\tau^s(\Lambda(g)).$$ Define
$$l_i(g)=(\pi_i(g))(K_i(g)),\quad i=1,2.$$
We can take $K_1(g)$ and $K_2(g)$ small so that
$$\frac{\|D_x\pi_i(g)\|}{\|D_y\pi_i(g)\|}\,\,\text{close to}\,\,1,\quad \text{for}\,\,x,y\in K_i,\,\,i=1,2,$$
which implies that
$$\tau(l_i(g)) \,\,\text{close to }\,\,\tau(K_i(g)),\quad \text{for}\,\,i=1,2.$$
Hence,  together with $\tau(K_1(g))\cdot \tau(K_2(g))>1$,  we have
$$\tau(l_1(g))\cdot \tau(l_2(g))>1.$$

For two Cantor sets $Y_1, Y_2$, let $I_1=[s_1,s_2], I_2=[t_1,t_2]$
be minimal closed intervals such that $I_1\supseteq Y_1,
I_2\supseteq Y_2$.  We say $Y_1, Y_2$ are linked if $I_1, I_2$ are
linked, i.e., $s_1<t_1<s_2<t_2$ or $t_1<s_1<t_2<s_2$.  Since
$l_1(f)$ and $l_2(f)$ has a boundary point in common, so taking a
small perturbation, there exists a $C^2$ open set
$\mathcal{N}\subset \mathcal{N}_2$ such that $l_1(g)$ and $l_2(g)$
are linked and $\tau(l_1(g))\cdot\tau(l_2(g))>1$, $\forall\,g\in
\mathcal{N}$. By Lemma \ref{gap lemma}, the third case of Lemma
\ref{gap lemma} is satisfied, which implies the existence of a
tangency $z_0\in l(g)$ of $\mathcal{F}^u_{g}(x_0,L+\delta_1)$ and
$\mathcal{F}^s_{g}(y_0,L+\delta_1)$ for some $x_0\in K_1(g)$,  $y_0\in
K_2(g)$. Moreover, one of the following two cases happens:

\begin{itemize} \item[(i)] there exist $u_i\in l_1(g)$ below $z_0$ with $u_i\to z_0$ as $i\to
+\infty$;\\
\item[(ii)] there exist $v_i\in l_2(g)$ above  $z_0$ with $v_i\to z_0$ as
$i\to+
\infty$.
\end{itemize}
Otherwise, $z_0$ is  a boundary point of both $l_1(g)$ and $l_2(g)$,
contradicting that $\Int(F_j\cap G_j)\neq \emptyset$, $\forall\,
j\in \mathbb{N}$, where $\{F_j\}, \{G_j\}$ are defining sequences of
$l_1(g), l_2(g)$ with $\tau(\{F_j\})\tau(\{G_j\})>1$.

\begin{Lem}\label{small intersection}Given $\delta>0$, there are $x_i\in K_i$, $i=1,2$,  such that $W^u_{L+\delta_1}(x_1)$
intersects $W^s_{L+\delta_1}(x_2)$ transversally in a
$\delta$-neighborhood of  $z_0\in l(g)$ and, the two nearby
intersections are  contained in $B(z_0,\delta)$.
\end{Lem}
\begin{proof}   We have assumed that $\mathcal{F}^u(z_0)$ stays
on the up-side of the horizontal $\mathcal{F}^s(z_0)$ in a small
neighborhood of $z_0$.

 Corresponding to (i), there is $u_i\in l_1$ on the below-side of
$z_0$ with $d(u_i, z_0)<(a_1^{-1}\delta)^2$, then we can take
$x_1\in K_1$ such that $u_i\in W_{L+\delta_1}^u(x_1)\cap l(g)$ and,
$W^u_{L+\delta_1}(x_1)$ transversally intersects
$W_{L+\delta_1}^s(y_0)$ in a $\delta$-neighborhood of some
$\widetilde{z}\in l(g)$; Let $x_2=y_0$.

Corresponding to (ii), there is $v_i\in l_2$ on the above-side of
$z_0$ with $d(v_i, z_0)<(a_1^{-1}\delta)^2$. The argument is
similar by taking $x_1=x_0$.
\end{proof}

Given $\delta>0$, let $x_1$, $x_2$ as in Lemma \ref{small
intersection}, and  $z_1\in W^u_{L+\delta_1}(x_1)\cap l(g)$, $z_2\in
W^s_{L+\delta_1}(x_2)\cap l(g)$, $d(z_1,z_2)<(a_1^{-1}\delta)^2$.
Since $g$ is $C^2$, the two maps
\begin{eqnarray*}x\in \Lambda(g) &\to&  x^u\in W^u_{L+\delta_1}(x)\cap l(g),\\[2mm]
x\in \Lambda(g) &\to&  x^s\in W^{s}_{L+\delta_1}(x)\cap l(g)
\end{eqnarray*}
are $C^1$ smooth and well defined in a neighborhood of  $x_1$ and a
neighborhood of $x_2$, respectively. We can take $a_2>0$ as the
Lipschitz constant for the above two maps. Applying Lemma \ref{bound
period} for $\vep=a_2^{-1}d(z_1,z_2)/3$, we can find a periodic
point $q\in \Lambda(g)$ satisfies \begin{itemize}\item $\tau(q)\in
[2|\log(a_2^{-1}d(z_1,z_2)/3)|/|\log\lambda|,\,\,
9|\log(a_2^{-1}d(z_1,z_2)/3)|/|\log\lambda|]$ \\
\item
$d(q,x_1)\leq a_2^{-1}d(z_1,z_2)/3,\quad d(g^{i_0}q,x_2)\leq
a_2^{-1}d(z_1,z_2)/3\quad \text{for some}\,\,i_0\in (0,\tau(q)).$
\end{itemize}
Furthermore,
$$d(q^u,z_1)\leq d(z_1,z_2)/3,\quad d((g^{i_0}q)^s,z_2)\leq d(z_1,z_2)/3.$$
Hence, $W^u_{L+\delta_1}(q)$ transversally intersects
$W^s_{L+\delta_1}(g^{i_0}q)$ at two points $y_1,y_2$ with
$d(y_1,y_2)\leq a_1\sqrt{\frac{5}{3}d(z_1,z_2)}$.

 We choose a  rectangle centered at the origin
$O:=(f^{i}(q))^s$ as follows
$$L_{z_1,z_2}=\big{\{}(e_1,e_2)\mid |e_1|_s\leq a_1\sqrt{\frac{5}{3}d(z_1,z_2)},\,|e_2|_u\leq \frac{d(z_1,z_2)}{10}\big{\}}.$$
where $|\cdot|_s$, $|\cdot|_u$ denote  the distances in the
horizontal axis ($s$-direction) and  the vertical axis,
respectively.

By iterations, $g^n(L_{z_1,z_2})$ will  become longer along
$u$-foliation, and narrower along $s$-foliation. Observe that for
$d(z_1,z_2)$ sufficiently small, $\tau(q)$ will be large enough. In order to make $g^{k\tau(q)+i_0}(L_{z_1,z_2})$ as
$u$-foliation intersect $L_{z_1,z_2}$ as $s$-foliation
transversally near $O$, we  take $k$ such that  the
length of the $u$-leaves of $ g^{k\tau(q)}(L_{z_1,z_2})$ is at least
$L+\delta_1$, i.e.,
$$\frac{d(z_1,z_2)}{10}\lambda^{-\tau(q)k}\geq  L+\delta_1. $$
  So,
$$k\leq \frac{\log\frac{d(z_1,z_2)}{10}-\log(L+\delta_1)}{\tau(q)\log\lambda}\leq
\frac{\log\frac{d(z_1,z_2)}{10}-\log(L+\delta_1)}{2|\log(a_2^{-1}d(z_1,z_2)/3)|/|\log\lambda|\cdot\log\lambda}.
$$
We can take a constant $T_1\in \mathbb{N}$ independent  of
$d(z_1,z_2)$, such that $k\leq T_1$.  For
$t=T_1\tau(q)+i_0\in[T_1\tau(q),(T_1+1)\tau(q)]$, $g^t(L_{z_1,z_2})$
will intersect $L_{z_1,z_2}$ transversally near $O$. Here we need to
further cut the unnecessary parts outside the foliation
$\mathcal{F}^u$. This is equivalent to take  a sub-rectangle
$L_{z_1,z_2}'\subset L_{z_1,z_2}$  with the height in the vertical
direction
 of  $L_{z_1,z_2}$ smaller but no change on the length in the horizontal direction. To make $g^t(L'_{z_1,z_2})$ also intersect $L'_{z_1,z_2}$
transversally near $O$, it is sufficient to let  the length of
$u$-leaves of $g^t(L'_{z_1,z_2})$  be $C_0\sqrt{d(z_1,z_2)}$ for
some constant $C_0$ independent of $d(z_1,z_2)$.
 Therefore,
  \begin{eqnarray*}\diam^s(g^n(L_{z_1,z_2}'))&\leq& a_0\cdot\diam^s(L_{z_1,z_2}')\leq a_0\cdot\sqrt{\frac{5}{3}a_1d(z_1,z_2)},\,\,0\leq n\leq
  t;\\[2mm]
\diam^u(g^n(L_{z_1,z_2}'))&=& \diam^u(g^{-(t-n)}\circ
g^t(L_{z_1,z_2}'))\\[2mm]&\leq& a_0 \diam^u(
g^t(L_{z_1,z_2}'))\leq a_0\cdot C_0\sqrt{d(z_1,z_2)},\,\, 0\leq
n\leq t.
  \end{eqnarray*}
Therefore, $ \diam(g^n(L_{z_1,z_2}'))\leq C_1 \sqrt{d(z_1,z_2)}$,
$0\leq n\leq t$,  for some constant $C_1$ independent of
$d(z_1,z_2)$.

\begin{figure}[h]
\begin{center}
\begin{picture}(200,120)(-30,-12)









\multiput(100,30)(2,0){28}{\line(1,0){0.5}}

\put(-45,35){\line(1,0){200}}

\put(100,38){\line(1,0){55}}

\multiput(100,43)(2,0){28}{\line(1,0){0.5}}

\multiput(100,30)(0,2){7}{\line(0,1){0.5}}

\multiput(155,30)(0,2){7}{\line(0,1){0.5}}

 {\color{red}

\put(5,-25){\line(0,1){100}}

\put(8,-25){\line(0,1){100}}





\qbezier(5,75)(5,95)(35,100)

\qbezier(35,100)(90,105)(110,50)

\qbezier(110,50)(130,0)(150,50)

\qbezier(2,75)(2,95)(32,103)

\qbezier(32,103)(87,108)(112,53)

\qbezier(112,53)(130,3)(147,53)

\qbezier(8,75)(8,95)(32,97)

\qbezier(32,97)(93,102)(108,47)

\qbezier(108,47)(130,-3)(153,47)

}

\put(2,-25){\line(0,1){100}  }

{\color{blue}\put(130,0){\line(0,1){50}}}



 \put(130,60){\makebox(0,0){\tiny $l(g)$}}

  \put(170,37){\makebox(0,0){\tiny $L'_{z_1,z_2}$}}

  \put(172,52){\makebox(0,0){\tiny $g^t(L'_{z_1,z_2})$}}

\put(62,85){\makebox(0,0){\tiny $g^t(L_{z_1,z_2})$}}

  \put(160,24){\makebox(0,0){\tiny $L_{z_1,z_2}$}}

 \put(153,47){\line(-1,1){3}}

 \put(108,47){\line(1,1){3}}

$$\, $$

\end{picture}
$$\, $$
\end{center}
\caption{Transversal intersections}
\end{figure}
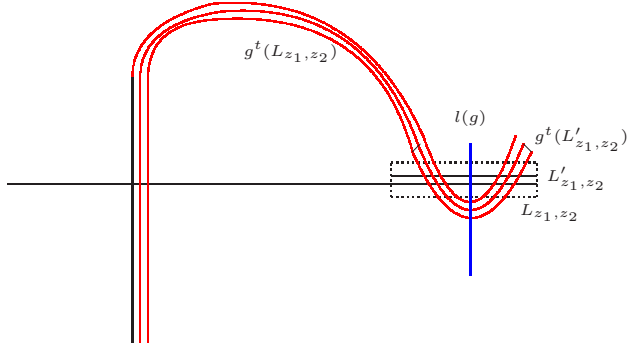

Let   $$\Gamma_g(z_1,z_2):=\bigcap_{n\in \mathbb{Z}}
g^n(L_{z_1,z_2}').$$ Then $\Gamma_g(z_1,z_2)$ is a periodic
hyperbolic  basic set with period $t$ and with diameter no more than
$C_1\sqrt{d(z_1,z_2)}$. Let $\eta=C_1\sqrt{d(z_1,z_2)}$. Since
$L'_{z_1,z_2}\cap g^t(L'_{z_1,z_2})$  contains two strips,
\begin{eqnarray*} h(g^t, \Gamma_g(z_1,z_2))\geq \log2,\end{eqnarray*}
which implies for the maximal entropy measure $\mu$ supporting on
$\Gamma_g(z_1,z_2)$,
\begin{eqnarray*}h_{\loc}(g,\mu,\eta)&\geq& h(g, \Gamma_g(z_1,z_2)) \geq
\frac{h(g^t, \Gamma_g(z_1,z_2))}{t}\\[2mm]&\geq&   \frac{\log2}{(T_1+1)\tau(q)}\geq
\frac{\log2}{9(T_1+1)|\log(a_2^{-1}d(z_1,z_2)/3)|/|\log\lambda|}\\[2mm]
&\geq& \frac{C_2}{|\log \eta|}\end{eqnarray*} for some constant
$C_2$ independent of $\eta$.  Hence, $h_{\loc}(g,\eta)\geq C_2/|\log
\eta|$.  Note that $a_i,C_i$, $i=0,1,2$, can be chosen uniformly for
$g\in \mathcal{N}$. The proof of Theorem \ref{non EE} is completed.
\hfill $\Box$
\begin{Rem}
The presence of horseshoes for surface diffeomorphisms with
homoclinic tangencies has previously been studied in a qualitative
way by Homburg and Weiss, see in \cite{Hom-Weiss}.
\end{Rem}

\section{$h$-expansiveness for endomorphisms on compact  homogenous Riemmanian manifold}

We prove here Theorem  \ref{homogenous}. Let $G/\Lambda$ be a compact  homogenous Riemannian manifold of dimension $m$.
There exists $\vep>0$ such that the map $\pi: G\rightarrow G/\Lambda$ given by $\pi(g):=g\Lambda$ for any $g\in G$ is a $\vep$-local isometry, i.e. the restrication of $\pi$ to any $\vep$-ball of $G$ is an isometry (see e.g. \cite{Bow71}). We  may also assume that the exponential map $\exp:\mathfrak{g}\rightarrow  G$ of the Lie group $G$ with Lie algebra $\mathfrak{g}$ is a diffeomorphism between the $\vep$-ball at $0\in \mathfrak{g}$ and the $\vep$-ball at the unity $e\in G$. Let $L>1$ be a Lipschitz constant of the diffeomorphism $\psi:=\pi\circ \exp |_{B(0_\mathfrak{g},\vep)}$.


We consider an endomorphism $\phi$ of $G/\Lambda$. Recall that $\phi$ is given  by $\underline{g}\in G$ and a morphism of group $\Phi:G\rightarrow G $ as follows $\phi(g\Lambda)=\underline{g}\Phi(g)\Lambda$ for any $g\Lambda\in G/\Lambda$.

 Assume firstly $\underline{g}=e$, then $\phi\circ \pi=\pi\circ \Phi$. We have also the following property of commutativity
$\Phi\circ \exp Y=\exp \circ d\Phi (Y)$ for any $Y\in \mathfrak{g}$ and thus $$\phi\circ\psi =\psi\circ d\Phi.$$
Fix $\delta>0$. We consider in $T_{e}G\simeq \mathfrak{g}$ the hyperplans defined by $\{x_i=+/- k\delta/L\sqrt{m}\}$ for some orthonormal system of coordinates $(x_1,...,x_m)$ and for $k$ being an integer with  $0\leq k\delta/L\sqrt{m}<\vep$. These hyperplans separate any two points with distance larger than $\delta/L$ in the ball of radius $\vep$ in $\mathfrak{g}$. In particular if $\mathfrak{g}'$ denotes the complementary set of these hyperplanes in $\mathfrak{g}$ any connected component of $\mathfrak{g}'$ intersected the $\varepsilon$-ball is contained in a $\delta/L$-ball.

Now for any integer $n$ the set $\bigcap_{0\leq
k<n}d\Phi^{-k}\mathfrak{g}'$ has a polynomial number in $n$ of
connected components : indeed $N$ hyperplanes separates
$\mathbb{R}^m$ in at most $CN^m$ components for some constants $C$
depending only on $m$. If such a connected component intersects the
Bowen ball for $d\Phi$ of size $\vep$ and length $n$ at
$0_\mathfrak{g}$,

$$B_n(d\Phi,0_\mathfrak{g},\vep):=\{X\in \mathfrak{g}, \ \|d\Phi^k(X)\|<\vep, \ 0\leq k<n  \},$$
then it is contained in a $(\delta/L,n)$-Bowen ball for $d\Phi$.
This argument applies to any linear map on an Euclidean space and is
used for example to study the entropy of (piecewise) affine maps
(see e.g. \cite{Buzzz}).

 As $\psi |_{B(0_{\mathfrak{g}},\vep)}$ is a diffeomorphism with Lipschitz constant less than or equal to $L$ restricted to the ball of radius  $\vep$ and by the relation $\phi\circ\psi =\psi\circ d\Phi$ we may cover the Bowen ball of radius $\vep$ at $e\Lambda$ for $\phi$ by a polynomial number of such Bowen balls of  arbitrarily small radius $\delta>0$. It follows that
$h_{top}(\phi,B_{\infty}(e\Lambda,\vep))=0$. As the metric is left invariant  it holds also in fact for  Bowen balls at any $g\Lambda\in G/\Lambda$ and therefore we conclude that

$$h^*(\phi,\vep)=0.$$

The general case $\underline{g}\neq e$ follows also immediately from the left invariance of the metric on $G/\Lambda$.

\section{$C^\infty$ examples with arbitrarily  slow convergence}
\subsection{Proof of Theorem \ref{thm2}}

Let $T:M\rightarrow M$ be a $C^\infty$ diffeomorphism with an
interval $I$ of homoclinic tangencies for some hyperbolic periodic
point $p$. At most taking some iteration we can suppose $p$ is a
fixed point of $T$.  We denote by $m$
 the dimension of $M$ and by $m_u$ and $m_s$ the dimensions of the unstable and stable manifolds tangent at $I$.
Assume that in a local chart $U\supset I$ the interval $I$ may be
written as $I=[0,1]$ and $U\supset [-3,3]^m$. For any
positive real function  $a:(0,1)\to \mathbb{R}^+$ with $\lim_{\vep\rightarrow 0} a(\vep)=0$ we will
construct a $C^\infty$ map $f_a:[-1,1]\rightarrow [0,1]$ such that if
$\theta_a$ is a $C^\infty$ diffeomorphism  of $M$ satisfying in local coordinates
\begin{itemize}\item $\theta_a=Id$ outside $[-2,2]^m$ ;\\[2mm]
 \item $\theta_a(x,y)=(x,y_1+f_a(x_1),y_2, \cdots, y_{m_u})$ for $(x,y)\in
[-1,1]^{m_s}\times [-1,1]^{m_u}$, \end{itemize}  then the
diffeomorphism $F_a:=\theta_a\circ T$ satisfies
$$h_{\loc}(F_a,\vep)\geq a(\vep)$$ for all $0<\vep\leq \zeta(f_a)$
with some constant $\zeta(f_a)>0$. The map $f_a$ is chosen to be
$C^\infty$ flat at $0$ so that $F_a$ has a homoclinic tangency of
infinite order at $(0,0)$. Moreover since $(x,y)\mapsto
(x,y_1+f_a(x_1),y_2, \cdots,y_{m_u})$ is volume preserving we may
choose $\theta_a$ be also volume preserving by the pasting Lemma of
Arbieto and Matheus (see Lemma 3.9 of \cite{past}).
 Also
$\theta_a$ is $C^\infty$ close to the identity when $f_a$ is
$C^\infty$ close to zero.

The idea introduced by Misiurewicz in \cite{Mis73} and developed
later by Downarowicz-Newhouse \cite{DN05} and Buzzi \cite{Buz13}  and
in other recent works \cite{burguet},\cite{cat},\cite{arbcie}
consists in creating arbitrarily  small horseshoes accumulating at
the fixed point $p$ by choosing the graph of $f$ looking like small
snakes closer and closer to the tangency.

We describe now the main properties of the map $f_a$.  Let $\chi$ be a non
negative $C^\infty$ bump function, such that $\chi(t)=1$ if $0\leq t\leq
1$ and  $\chi(t)=0$ if $t> 2$ or $t<-1$.  We  produce
snakes only on the intervals of the form
$[c_n,d_n]:=\left[\frac{1}{4n+1},\frac{1}{4n}\right]$ for all integers $n$.
More precisely we put with $R_{a,n}>M_{a,n}>0$ and $N_{a,n}\in \NN$ (which we precise later on),
$$f_a=\sum_nf_n, \ \textrm{with } f_n:=\chi\left(\frac{x-c_n}{d_n-c_n}\right)\left(R_{a,n}+M_{a,n}\sin\left(N_{a,n}\frac{x-c_n}{d_n-c_n}\right)\right).$$

This sum is zero on $\mathbb{R}^-$ and it defines a $C^\infty$
function on $\mathbb{R}^+\setminus \{0\}$ as the terms of the sum
are $C^\infty$ function with  disjoint compact supports accumulating
only at $0$. We let $\vep=\vep_n:=d_n-c_n=\frac{1}{4n(4n+1)}$ and we
denote $R_{a,\vep}:=R_{a,n}$, $M_{a,\vep}:=M_{a,n}$,
$N_{a,\vep}:=N_{a,n}$ and $f_\vep:=f_n$ for the integer $n=n_\vep$
giving $\vep$. We may choose $R_{a,\vep}$ and $M_{a,\vep}$
 so that any
branch of the sinusoidal in the graph of $f_a$ crosses all the
branches after a time $P_{a,\vep}$ with
$$M_{a,\vep}e^{\lambda_u(p)P_{a,\vep}}=\vep,$$
$$R_{a,\vep}e^{\lambda_u(p)P_{a,\vep}}\leq C,$$ where $\lambda_u(p)>0$ is the minimum of absolute values of all Lyapunov
exponents  of $T$ at $p$ and $C=C(T)$ depends only on $T$.

We consider a  rectangle $L_{a,\vep}$ as in the proof of Theorem
\ref{non EE}, Page 36. Here the intersection of $L_{a,\vep}$ with
$F_a^{P_{a,\vep}}L_{a,\vep}$ consists in $N_{a,\vep}$ strips so that
the entropy of the associated horseshoe $H_{a,\vep}$ is given by
\begin{eqnarray*}
h(H_{a,\vep})=\frac{\log N_{a,\vep}}{P_{a,\vep}}.
\end{eqnarray*}

Note also that  $H_{a,\vep}$ is contained in an infinite dynamical
ball of size $\vep$. We explain now how to choose  $P_{a,\vep}$ and
$N_{a,\vep}$ with respect to $\vep$ and $a$. Firstly we can take
$P_{a,\vep}=-\log \vep/a(\vep)$ and then $N_{a,\vep}$ as follows
$$N_{a,\vep}:=\lceil e^{P_{a,\vep}a(\vep)}\rceil=\lceil 1/\vep\rceil.$$
It will imply by considering  a measure of maximal entropy of $H_{a,\vep}$ that

$$h_{\loc}(F_a,\vep)\geq \frac{\log N_{a,\vep}}{P_{a,\vep}} \geq a(\vep).$$

The only thing we
need to check is that the resulting map may be $C^\infty$ extended at $0$.
 It is enough to prove that
$\|f_\vep\|_r$ goes to zero when $\vep$ goes to zero for any given
integer $r$. Fix $r\geq 1$. We have for some  constant $C_r=C(r,T)$
\begin{eqnarray*}
\|f_\vep\|_r& \leq  &C_r
M_{a,\vep}\left(N_{a,\vep}/\vep\right)^{r}\\[2mm]
&\leq & C_r\vep
e^{-\lambda_u(p)P_{a,\vep}}\left(e^{a(\vep)P_{a,\vep}}/\vep\right)^r\\[2mm]
&=& C_r(1/\vep)^{-\frac{\lambda_u(p)}{a(\vep)}+2r-1}.
\end{eqnarray*}
This last term goes to zero when $\vep $ goes to zero because
$\lambda_u(p)>0$ and $\lim_{\vep\rightarrow 0}a(\vep)=0$. This
proves $F_a$ may be $C^\infty$ extended. Moreover,  when defining
$f_a$ we only consider the rest of the series from $N$, i.e.
$$f_a^N:=\sum_{n\geq N}f_n,$$ then the resulting diffeomorphisms
$(F_a^N)_N$ converge to $T$ in the $C^\infty$ topology when $N$ goes
to infinity. \hfill $\Box$

\begin{Rem}
The construction  may be easily adapted to interval maps to
produce examples with the same properties. One only needs to repeat the previous construction near
a flat homoclinic tangency at a hyperbolic repulsing fixed point of a $C^\infty$ interval map
(See for example \cite{burguet} and \cite{ruette} for similar constructions of horseshoes accumulating near a tangency).
\end{Rem}

\begin{Rem}\label{upper}
In the above proof the local entropy is produced by small
horseshoes, i.e. horseshoes included in some infinite
$\vep$-dynamical ball, which are persistent under $C^1$ (even $C^0$
for interval maps)
 small perturbation. In particular if  $f_a$ is as in the above example  there exists for any $\vep$ a polynomial map $P_\vep$
 with $h^*(P_\vep,2\vep)\geq a(\vep)$. That's why one can not expect to have a lower bound in $C/|\log\vep|$
 with $C$ independent from the degree in Theorem \ref{quasionedim} or Corollary \ref{polynomial}.
\end{Rem}

\begin{Rem}In \cite{DN05}, \cite{cat}, \cite{arbcie}, the persistence of homoclinic tangencies and of  small horseshoes allow
to use a Baire argument to build generic examples with no principal symbolic extensions, in particular non asymptotically $h$-expansive.
 Here we do
not know if Corollary \ref{thm22} holds for a $C^\infty$ generic
subset of Newhouse domains. Indeed we only are able to show that
$\{f:  \ h^*(f,\vep)\geq a(\vep)\}$ are $\tilde{a}(\vep)$-dense in
Newhouse domains for some function $\tilde{a}$ depending on $a$.
\end{Rem}

\subsection{Proof of Theorem \ref{exa}}
We will use the same construction as in the proof of Theorem \ref{thm2}.
For any positive $r$ we let $D_r\subset \mathbb{R}^2$ be the $2$-disks of radius $r$  at zero.
 We first consider an analytic flow $(F_t)_t$ of the plane with  a homoclinic orbit  $\Gamma\subset  D_{1/2}$  at
some hyperbolic fixed point $p$.  The stable/unstable manifolds of
$p$  are analytic (in fact it follows from the Irwin Method for
invariant manifolds \cite{Irwin} \cite{aba}  and the implicit
function theorem for ultradifferentiable maps in \cite{Kom} that the
stable/unstable manifolds at a hyperbolic fixed point of a smooth
system are in the same ultradifferentiable class  as the system).
For some analytic metric,  a piece $I$ of the homoclinic orbit and
$U$ a neighborhood of $I$ may be written
 through the exponential map as  $I=[0,1]\times \{0\}$ and $U$ as a neighborhood of $[-3,3]^2$. Note this interval $I$ is by construction
  an interval of homoclinic tangencies.
Then in any non quasianalytic $U$-ultradifferentiable class  one can find a bump function $\psi$
supported in $D_{3/4}$, $0\leq \psi\leq 1$ and with $\psi=1$ on $D_{2/3}$. Finally we consider
 the smooth system $T$ defined as the time $\psi$ map of the flow, i.e. $T(x)=F_{\psi(x)}(x)$
 for all $x\in D_1$. As $U$-ultradifferentiable maps are closed under composition \cite{bier} the
 diffeomorphism $T$ of $D_1$ may be chosen in any non quasi analytic class. Finally observe that
  $T$ coincides with  the identity in a neighborhood of the boundary of $D_1$ and that  $\Gamma$
  is also an homoclinic orbit for $T$ at the $T$- hyperbolic fixed point $p$. One can also choose $T$
   with $\|DT(p)\|=\|DT\|_\infty$ and $\|DT\|_\infty>1$ arbitrarily close to one (in the following we will
   take $T$ with $2BD\log \|DT\|<1$ for some constants $B$ and $D$ given later on).

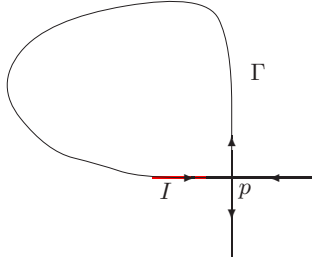
\begin{figure}[h]
\begin{center}
\begin{picture}(200,80)(-70,-12)

\put(-10,0){\line(1,0){40}}

{\color{red}\put(-30,0){\line(1,0){20}}}

\spline(0,0)(-35,0)(-50,5)(-70,10)(-90,35)(-70,60)(-10,70)(0,50)(0,5)(0,0)

\put(0,-30){\line(0,1){40}}

\put(0,-15){\vector(0,-1){}}

\put(0,15){\vector(0,1){}}

\put(-15, 0){\vector(1,0){}}

\put(15, 0){\vector(-1,0){}}

\put(5,-5){\makebox(0,0){\small $p$}}

\put(-25,-5){\makebox(0,0){\small $I$}}

\put(10,40){\makebox(0,0){\small $\Gamma$}}

$$\, $$

\end{picture}
$$\, $$
\end{center}
\caption{Interval $I$ of tangencies from a homoclinic orbit
$\Gamma$}
\end{figure}

 We fix some $\vep>0$ and we consider the map $F_a^\vep$ obtained as in the previous construction
 in the proof of Theorem \ref{thm2} but where we change $f_a$ by considering  only  the $n_{\vep}$-term of the series, i.e.
$$f_a:=f_{\vep}.$$  Moreover we may choose
$\theta_a$ in the local chart given by the exponential map by
$\theta_a(x,y)=(x,y+\chi(x)\chi(|y|)f_a(x))$ for $(x,y)\in
[-3,3]^2\subset U$ with $\chi$ a bump function as in the previous
proof. Then to make explicit computations we will take $T$ and
$\chi$ in the non quasi-analytic ultradifferentiable class
$(k^{2k})_k$, i.e. $T\in \Diff_U^{(k^{2k})_k}(D_1)$ and $\chi\in
U^{(k^{2k})_k}(\mathbb{R},\mathbb{R})$.

Let $\mathcal{M}=(M_k)_k$ be the weight defined  by
$$M_0=\|DT\|$$
and for all integers $k>0$
$$M_k:=M_0\left(\frac{1}{a^{-1}(\log\|DT\|/k)}\right)^k.$$

As  $1/a(e^{-.}):(0,+\infty)\to \mathbb{R}^+$ is a concave
nondecreasing function, its inverse $\log
\left(\frac{1}{a^{-1}(1/x)}\right)$ is convex  nondecreasing.  It
follows that $(\frac{\log(M_k/M_0)}{k})_k$ is convex nondecreasing,
which implies that $(\log M_k)_k$ is convex, i.e. $\cM$ is a
logarithmic convex weight.
  The condition $a(\vep)\geq \vep^{1/7}\geq 2BD\log\|DT\|\vep^{1/7}$ for all $\vep$ implies that $\left(M_k/M_0\right)^{\frac{1}{k}}=\frac{1}{a^{-1}(\log\|DT\|/k)}\geq (2BDk)^{7}$ for all $k$.  In particular $(\log (M_k/M_0)/k)_k$ is
not bounded.   We also consider the weight $\tilde{\mathcal{M}}=(\tilde{M}_k)_k$  defined for
all integers $k$ by
$$\tilde{M}_k:=(2BDk)^{7k}M_k^2/M_0.$$
The weight $\tilde{\mathcal{M}}$ is also clearly  logarithmic convex. Observe also that
$\tilde{M}_k/\tilde{M}_0\leq (M_k/M_0)^3$.

We have for all $x>0$ and for all $0<\vep<1$

\begin{eqnarray*}
G_{\tilde{\mathcal{M}}}(x)&\geq &G_{\mathcal{M}}(x/3)\\[2mm]
&\geq &
\frac{\log\|DT\|}{a(1/e^{x/3})},\\[2mm]
G_{\tilde{\mathcal{M}}}(3|\log\vep|)&\geq
& \frac{\log\|DT\|}{a(\vep)}.
\end{eqnarray*}

We check now for any $\vep>0$ that $F_a^\vep$ belongs to  $C_V^{\tilde{\mathcal{M}}}(D_1)$.
 Let us  compute $\|F_a^{\vep}\|_{r}$ for any $r$. We put $r_{\vep}:=
\lambda_u(p)/a(\vep)$. By applying Faa di Bruno formula for
 $F_a^{\vep}=T\circ \theta_a^{\vep}$ as in the proof of Theorem \ref{thm2} we have for any
$r\in\NN$:
\begin{eqnarray*}
\|F_a^{\vep}\|_{r} &\leq & B_{r}\|T\|_{r} \max_{\sum i j_{i}=r} \|\theta_a^\vep\|_{i}^{j_i}.
\end{eqnarray*}
Now by derivation of a product we have
$$\|\theta_a^{\vep}\|_i\leq 2^i\|\chi\|_i \|f_{\vep}\|_i.$$
As  $\chi$ is in  $V^{(B^kk^{2k})_k}$ for some constant $B\geq 1$ we have with $C=C(T)$

\begin{eqnarray*}
 \|f_\vep\|_i&\leq & C2^iM_{a,\vep}\|\chi\|_i
 (N_{a,\vep}/\vep)^{i}\\[2mm]
&\leq & C(2B)^ii^{2i}M_{a,\vep}(N_{a,\vep}/\vep)^{i}\\[2mm]
\|\theta_a^{\vep}\|_i&\leq
&C(2Bi)^{4i}M_{a,\vep}(N_{a,\vep}/\vep)^{i}
\end{eqnarray*} and then, as $T$ is in $C^{(k^{2k})_k}$, we have for some constant
$D=D(T)\geq 1$,
\begin{eqnarray*}
\|F_a^{\vep}\|_{r} &\leq & CD^rr^{3r} \max_{\sum ij_i=r}
\|\theta_a^\vep\|_{i}^{j_i}\\[2mm]
&\leq & C(2BDr)^{7r} M_{a,\vep}(N_{a,\vep}/\vep)^{r}\\[2mm]
&\leq &
C(2BDr)^{7r}(1/\vep)^{-\frac{\lambda_u(p)}{a(\vep)}+2r-1}\\[2mm]
&\leq & C(2BDr)^{7r} (1/\vep)^{-r_\vep+2r-1}.
\end{eqnarray*}
and thus for any $r$ and $\vep$ small enough ($C\vep \leq 1$) we get
\begin{eqnarray*}
\log \|F_a^\vep\|_{r}/r&\leq & 7\log(2BDr)+2\log (1/\vep)\\[2mm]
&\leq & 7\log(2BDr)+2\log (1/a^{-1}(\log\|DT\|/r_\vep))\\[2mm]
&\leq & 7\log(2BDr)+2\log (1/a^{-1}(\log\|DT\|/r))\\[2mm]
&\leq & 7\log(2BDr)+\frac{2\log  (M_r/M_0)}{r},
\end{eqnarray*}
that is,   $F_a^\vep\in C_V^{\tilde{\cM}}(D_1)$. Finally we have
again \begin{eqnarray*}
h_{\loc}(F_a^\vep)&\geq &h(H_{a,\vep})\\[2mm]
&\geq &a(\vep)\\[2mm]
&\geq&\frac{\log\|DT\|}{G_{\tilde{\mathcal{M}}}(3|\log \vep|)}.
\end{eqnarray*}

The previous construction  may be embedded  in any manifold $M$ of
dimension larger than two to get diffeomorphisms  $(f_a)^M$ and
$(f_a^\vep)^M$  on $M$ with the required properties, as usually done
by embedding  $D_1\times D_1^{m-2}$ in a given $m$-dimensional
manifold and by extending $(f_a)^{D_1}\times Id_{D_1^{m-2}}$ and
$(f_a^\vep)^{D_1}\times Id_{ D_1^{m-2}}$ by the identity outside
$D_1\times D_1^{m-2}$, where $D_1^{m-2}$ is the unit disk centered
at zero in $\mathbb{R}^{m-2}$. \hfill $\Box$


\begin{thebibliography}{10}

\bibitem{aba}A. Abbondandolo and  P. Majer,
\newblock On the global stable manifold,
\newblock {\em Studia Math.}, {\bf 177}: 113--131, 2006.

\bibitem{AKM}R. L. Adler, A. G. Konheim,  and M. H. McAndrew,
\newblock Topological entropy,
\newblock{\em Trans. Am. Math. Soc.}, {\bf 114}: 309--319, 1965.



\bibitem{arbcie}A. Arbieto, A. Armijo, T. Catalan,  and   L. Senos,
\newblock Symbolic extensions and dominated splittings for generic $C^1$ diffeomorphisms,
\newblock {\em Math. Zeit.}, to appear.

\bibitem{past}A. Arbieto and   C. Matheus,
\newblock A pasting lemma and some applications for conservative systems,
\newblock {\em Ergod. Th. Dynam. Sys.}, {\bf 27}: 1399--1417, 2007.

\bibitem{bier} E. Bierstone and   P. D. Milman,
\newblock Resolution of singularities in Denjoy-Carleman classes,
\newblock {\em Selecta Math. (N.S.)}, {\bf 10}: 1--28, 2004.

\bibitem{Block}L. Block,
\newblock Homoclinic points of mappings of the interval,
\newblock {\em Proc. Am. Math. Soc.}, {\bf 72}: 576--580, 1978.


\bibitem{Bow71}R. Bowen,
\newblock Entropy for group endomorphisms and homogeneous spaces,
\newblock {\em Trans. Am. Math. Soc.}, {\bf 153}: 401--414, 1971.

\bibitem{Bow72b}R. Bowen,
\newblock Entropy expansive maps,
\newblock {\em Trans. Am. Math. Soc.}, {\bf 164}: 323--331, 1972.


\bibitem{Bowen}R. Bowen,
\newblock Equilibrium states and ergodic theory of Anosov
diffeomorphisms,
\newblock {\em Lect. Notes Math}. 470. New York:
Springer 1975.



\bibitem{Bur08}D. Burguet,
\newblock A proof of Yomdin-Gromov's algebraic lemma,
\newblock {\em Israel J. Math.}, {\bf 168}:  191--236, 2008.

\bibitem{burguet}D. Burguet,
\newblock Examples of $C^r$ interval map with large symbolic extension entropy,
\newblock {\em Disc. Cont. Dyn. Systems A}, {\bf 26}: 878--899, 2010.

\bibitem{burex}D. Burguet,
\newblock Symbolic extensions for nonuniformly entropy expanding maps,
\newblock {\em Colloq. Math.}, {\bf 121}: 129--151, 2010.

\bibitem{paperinpreparation}D. Burguet,
\newblock Complexity in the Algebraic Lemma,
\newblock In preparation.

\bibitem{burfis}D. Burguet and  T. Fisher,
\newblock Symbolic extensions for partially hyperbolic systems with $2$-dimensional center bundle,
\newblock {\em Disc. Cont. Dyn. Systems A}., {\bf 33}: 2253--2270, 2013.



\bibitem{Buz97}J. Buzzi,
\newblock Intrinsic ergodicity for smooth interval maps,
\newblock {\em Israel J. Math}., {\bf 100}: 125--161, 1997.

\bibitem{Buzzz}J. Buzzi,
\newblock  Intrinsic ergodicity of affine maps in $[0,1]^d$,
\newblock  {\em Monatshefte fur Mathematik}, {\bf 124}: 97--118, 1997.


\bibitem{Buz13}J. Buzzi,
\newblock $C^r$ surface diffeomorphism with no maximal entropy measure,
\newblock  {\em Ergod. Th. Dynam. Sys.},  to appear.

\bibitem{cat}T. Catalan,
\newblock A $C^1$ generic condition for existence of symbolic extensions of volume preserving diffeomorphisms,
\newblock {\em Nonlinearity}., {\bf 12}: 3505--3525, 2012.


\bibitem{Dows}
T. Downarowicz,
\newblock  Entropy structure,
\newblock {\em  J. Anal. Math.}, {\bf 96}:  57--116, 2005.

\bibitem{DM09}T. Downarowicz and    A. Maass,
\newblock  Smooth interval maps have symbolic extensions,
\newblock  {\em Invent. Math}., {\bf 176}: 617--636, 2009.



\bibitem{DN05}T. Downarowicz and   S. Newhouse,
\newblock  Symbolic extensions and smooth dynamical systems,
\newblock  {\em Invent. math}., {\bf 160}: 453--499, 2005.



\bibitem{Turaev} S. V. Gonchenko, D. Turaev,  and   L. Shilnikov,
\newblock Homoclinic tangencies of arbitrarily high order in conservative and dissipative two-dimensional maps,
\newblock {\em Nonlinearity}., {\bf 20}: 241--275, 2007.


\bibitem{Gr85}M. Gromov,
\newblock Entropy, homology and semialgebraic geometry,
\newblock {\em S\'eminaire N. Bourbaki}., {\bf 663}: 225--240, 1985--1986.


\bibitem{Hom-Weiss}
A. J. Hombourg and   H.  Weiss,
\newblock  A  geometric criterion for positive topological entropy. II.  Homoclinic
tangencies,
\newblock {\em Comm. Math. Phys}., {\bf 208}:   267--273, 1999.


\bibitem {HP} M. Hirsch and   C.  Pugh,
\newblock  Stable manifolds and hyperbolic sets,
\newblock  Global Analysis (Proc. Sympos. Pure Math., Vol. XIV, Berkeley, Calif., 1968) pp. 133--163 Amer. Math. Soc., Providence,
R.I., 1970.

\bibitem{Irwin}M. C. Irwin,
\newblock On the stable manifold theorem,
\newblock  {\em Bull. London Math. Soc.}, {\bf 2}: 196--198, 1970.



\bibitem{Koz}O.  Kozlovski, W. Shen, and S.  van Strien,
\newblock Density of hyperbolicity in dimension one,
\newblock {\em Ann. of Math.}, {\bf 166}: 145--182, 2007.


\bibitem{Liao} G. Liao, Uniform local entropy for analytic maps,
arXiv:1304.7601.


\bibitem{LVY} G. Liao, M. Viana, and J. Yang,
\newblock  Entropy conjecture for diffeomorphisms away from tangen-
cies,
\newblock  {\em  J. Eur. Math. Soc.},  {\bf 6}:  2043--2060, 2013.

\bibitem{Mane}R. Ma\~n\'e,
\newblock Expansive diffeomorphisms,
\newblock {\em Lecture Notes in Math.}, {\bf 148 }: 162--174, 1975.







\bibitem{Mis}M. Misiurewicz,
\newblock On non-continuity of topological entropy,
\newblock {\em Bull.
Acad. Pol. Sci., Ser. Sci. Math. Phys. Astron},  {\bf 19}:
319--320, 1971.


\bibitem{Mis73}M. Misiurewicz,
\newblock Diffeomorphim without any measure of maximal entropy,
\newblock {\em Bull. Acad. Pol. Sci.}, {\bf 21}: 903--910, 1973.

\bibitem{Mis76}M. Misiurewicz,
\newblock Topological conditional entropy,
\newblock {\em Studia Math.}, {\bf 2}: 175--200, 1976.


\bibitem{Misch}
M. Misiurewicz and W. Szlenk,
\newblock  Entropy of piecewise monotone mappings,
\newblock {\em Studia Math.}, {\bf 67}: 45--63, 1980.


\bibitem{New79}S. Newhouse,
\newblock The abundance of wild hyperbolic sets and nonsmooth stable
sets for diffeomorphisms,
\newblock {\em Publ. IHES.},
 {\bf 50}:  101--151, 1979.



\bibitem{New88}S. Newhouse,
\newblock Entropy and volume,
\newblock {\em Ergod. Th. Dynam. Sys.},
{\bf 8$^*$}:  283--299, 1988.


\bibitem{New89}S. Newhouse,
\newblock Continuity properties of entropy,
\newblock {\em Ann. of  Math.}, {\bf 129}: 215--235, 1989.

\bibitem{PT93}J. Palis and  F. Takens,
\newblock Hyperbolicity and sensitive chaotic
dynamics at homoclinic bifurcations,  Fractal dimensions and
infinitely many attractors,
\newblock Cambridge Studies in Advanced
Mathematics, {\bf 35}, Cambridge University Press, 1993.

\bibitem{PW06}J. Pila and  A. J. Wilkie,
\newblock The rational points of a definable set,
\newblock {\em Duke Math. J.}, {\bf 133}:
591--616, 2006.


\bibitem{Ruelle}D. Ruelle,
\newblock An inequality for the entropy of
differential maps,
\newblock{\em Bol. Soc. Bras. de Mat}., {\bf  9}:
83--87, 1978.


\bibitem{ruette} S. Ruette,
     \newblock Mixing $C^r$ maps of the interval without maximal measure,
     \newblock {\em Israel J. Math.}, \textbf{127}:  253--277, 2002.

\bibitem{Shub}M. Shub,
\newblock Dynamical systems, filtrations and entropy,
\newblock {\it Bull. Amer. Math. Soc.}, {\bf 80}: 27--41, 1974.




\bibitem{Smale}S. Smale,
\newblock  Differentiable dynamical systems,
\newblock {\em Bull. Amer. Math. Soc}., {\bf 73}: 747--817,  1967.


\bibitem{Sternberg}S. Sternberg,
\newblock  Local contractions and a theorem of Poincar\'e,
\newblock {\em Amer. J. Math}., {\bf 79}: 809--824,  1957.



\bibitem{Walter}P. Walters,
     \newblock  An introduction to ergodic theory
     \newblock   Graduate Texts in Mathematics, 79. Springer-Verlag, New York-Berlin, 1982.


\bibitem{Kom}T. Yamanaka,
\newblock Inverse map theorem in the ultra-F-differentiable class,
\newblock {\em Proc. Japan Acad.}, {\bf 65}: 199--202, 1989.

\bibitem{Yom87}Y. Yomdin,
\newblock Volume growth and entropy,
\newblock {\em Israel J. Math.}, {\bf 57}:  285--300, 1987.

\bibitem{Yom87-2}Y. Yomdin,
\newblock $C^k$-resolution of semialgebraic mappings. Addendum to:
``Volume growth and entropy'',
\newblock {\em Israel J. Math.}, {\bf 57}:  301--317, 1987.



\bibitem{Yom91}Y. Yomdin,
\newblock  Local complexity growth for iterations of analytic mappings and the semicontinuity moduli of
entropy,
\newblock {\em Ergod. Th. Dynam. Sys.},  {\bf 11}: 583--602, 1991.\\






\end{thebibliography}
\end{document}